\documentclass[12pt]{amsart}
\usepackage{enumerate,amssymb,amsfonts,amsmath,amsthm,amsrefs,dsfont,hyperref,xcolor,comment}
\usepackage[letterpaper, portrait, margin=28mm]{geometry}

\newtheorem{thm}{Theorem}[section]
 \newtheorem{cor}[thm]{Corollary}
 \newtheorem{lem}[thm]{Lemma}
 \newtheorem{prop}[thm]{Proposition}
 \theoremstyle{definition}
 
 \theoremstyle{remark}
 \newtheorem{rem}[thm]{Remark}
 \newtheorem{ex}[thm]{Example}
 \newtheorem{question}[thm]{Question}
 \numberwithin{equation}{section}

\newcommand{\defeq}{\ensuremath{\mathop{:}\!\!=}}

\newcommand{\1}{\mathds{1}}
\newcommand{\0}{\mathds{O}}

\newcommand{\Q}{\mathbb{Q}}
\newcommand{\R}{\mathbb{R}}

\newcommand{\N}{\mathbb{N}}

\newcommand{\Bo}{\mathcal{B}}

\newcommand{\8}{\infty}

\newcommand{\spa}{\mathrm{span}}

\newcommand{\Sol}{\mathrm{Sol}}
\newcommand{\supp}{\mathrm{supp}}
\newcommand{\Co}{\mathcal{C}}
\newcommand{\Fo}{\mathcal{F}}

\newcommand{\Ho}{\mathcal{H}}

\newcommand{\abs}[1]{\lvert#1\rvert}
\newcommand{\norm}[1]{\lVert#1\rVert}
\newcommand{\bigabs}[1]{\bigl\lvert#1\bigr\rvert}

\newcommand{\term}[1]{{\textit{\textbf{#1}}}}   
\newcommand{\marg}[1]{\marginpar{\tiny #1}}     
\renewcommand{\mid}{\::\:}

\newcommand{\goeseta}{\xrightarrow{\eta}}	
\newcommand{\goesl}{\xrightarrow{\lambda}}	
	
\newcommand{\goesetaB}{\xrightarrow{\eta_{\mathcal B}}}
\newcommand{\goesueta}{\xrightarrow{\mathrm{u}\eta}}	
\newcommand{\goesuIe}{\xrightarrow{\mathrm{u}_I\eta}}	
\newcommand{\goeso}{\xrightarrow{\mathrm{o}}}	

\newcommand{\goesso}{\xrightarrow{\sigma\mathrm{o}}}	
\newcommand{\goesu}{\xrightarrow{\mathrm{ru}}} 
\newcommand{\goesuo}{\xrightarrow{\mathrm{uo}}}	
\newcommand{\goesmuB}{\xrightarrow{\mu_\mathcal B}} 
\newcommand{\goesseta}{\xrightarrow{\mathrm{s}\eta}} 

\DeclareSymbolFont{bbold}{U}{bbold}{m}{n}
\DeclareSymbolFontAlphabet{\mathbbold}{bbold}

\DeclareMathOperator{\Span}{span}

\renewcommand{\le}{\leqslant}
\renewcommand{\ge}{\geqslant}

\makeatletter
\@namedef{subjclassname@2020}{\textup{2020} Mathematics Subject Classification}
\makeatother

\begin{document}

%
%
%
%
%
%
%
%
%

\title[Countability conditions in locally solid convergence spaces]
 {Countability conditions in locally solid convergence spaces}

\author{Eugene Bilokopytov}
\email{bilokopy@ualberta.ca}
\address{Department of Mathematical and Statistical Sciences,
  University of Alberta;}

\author{Viktor Bohdanskyi}
\email{vbogdanskii@ukr.net}
\address{Department of Analysis and Probability, Ukrainian National Technical University ``Igor Sikorsky Kyiv Polytechnic Institute'' (NTUU KPI);}

\author{Jan Harm van der Walt}
\email{janharm.vanderwalt@up.ac.za}
\address{Department of Mathematics and Applied Mathematics, University
  of Pretoria, Pretoria, South Africa}

\subjclass[2020]{Primary 46A19; Secondary 46A40, 54D55}

\keywords{Convergence structures, vector lattices, order convergence}


\begin{abstract}
We study (strong) first countability of locally solid convergence structures on Archimedean vector lattices.  Among other results, we characterise those vector lattices for which relatively unform-, order-, and $\sigma$-order convergence, respectively, is (strongly) first countable. The implications for the validity of sequential arguments in the contexts of these convergences are pointed out.
\end{abstract}

\maketitle
\section{Introduction}

The theory of convergence structures, as set out in \cite{Beattie:02} and \cite{OBrien:23}, provides a general setting for concepts of convergence of nets, filters and sequences which do not necessarily correspond to convergence with respect to a topology.  In the theory of vector lattices there are several natural and important examples of such non-topological convergences.  For instance, for a vector lattice $X$ there exists a (linear) topology $\tau$ on $X$ so that the order convergent nets in $X$ are precisely the $\tau$-convergent nets if and only if $X$ is finite dimensional, see \cite{Dabboorasad:20}; for relatively uniform convergence this is the case if and only if $X$ has a strong order unit, see \cite[Theorem 5]{Dabboorasad:18}.  Motivated by these and other examples, a general theory of locally solid convergence structures is developed in \cite{Bilokopytov:24}, see also \cite{VanderWalt:13}.  In this paper we consider countability conditions in the context of locally solid convergence structures.  Our main results characterise those vector lattices for which order convergence, $\sigma$-order convergence, relative uniform convergence and their unbounded modifications are (strongly) first countable.

The paper is organised in as follows.  Section \ref{sec:prelim} provides general background material on convergence structures.  In particular, sequential, Fr\'{e}chet-Urysohn convergences and (strongly) first countable spaces are discussed in Section \ref{Sec:  Sequential}, and Section \ref{Sect:Bornological} deals with countability properties of bornological convergence and the bounded modification of a convergence structure. Section \ref{Sec:  Locally solid} is dedicated to countability properties of the unbounded- and locally solid modifications of a convergence space. The heart of the paper is Section \ref{Sec: order}:  Here we characterise those vector lattices for which order convergence, $\sigma$-order convergence, and, unbounded order convergence, respectively, is (strongly) first countable.  Finally, in Section \ref{Sec:  Sequential arguments} we discuss the validity of sequential techniques in the presence of (strong) first countability.

\section{Generalities}\label{sec:prelim}

\subsection{Convergence structures}

Throughout this section $X$ denotes a nonempty set, unless otherwise stated.  By a \term{net in $X$} we mean function $x:A\to X$ from a pre-ordered, upward directed set $A$ into $X$; we denote such net by $(x_\alpha)_{\alpha\in A}$.  In keeping with \cite{Bilokopytov:24} we adopt the following notational convention:  the index set of a net is denoted by the capital Greek letter corresponding to the lower-case Greek letter which is used as the subscript for the terms of the net.  Hence the net $(x_\alpha)_{\alpha\in A}$ is denoted simply by $(x_\alpha)$.  A \term{tail} of a net $(x_\alpha)$ is any set of the form $\{x_\alpha\}_{\alpha\ge \alpha_0}$ with $\alpha_0\in A$.  A net $(y_\beta)$ is called a \term{quasi-subnet} of the net $(x_\alpha)$ if every tail of $(x_\alpha)$ contains a tail of $(y_\beta)$.  Note that every subnet of $(x_\alpha)$ is a quasi-subnet of $(x_\alpha)$, but not conversely.

A \term{filter} on $X$ is a nonempty collection $\mathcal{F}$ of nonempty subsets of $X$ which is closed under the formation of finite intersections and supersets.  A \term{filter base} on $X$ is a nonempty collection $\mathcal{B}$ of nonempty subsets of $X$ which is downward directed with respect to inclusion.  Every filter base $\mathcal{B}$ is contained in a smallest filter, $[\mathcal{B}]\defeq \{F\subseteq X \mid B\subseteq F \text{ for some } B\in \mathcal{B}\}$.  We note that if $\mathcal{B}$ is a filter base and $\mathcal{F}$ is a filter so that $\mathcal{F}\cap \mathcal{B}$ is a filter base, then $[\mathcal{F}\cap \mathcal{B}]\subseteq \mathcal{F}\cap [\mathcal{B}]$; in general the inclusion may be strict.  For instance, let $X=\R$, let $\mathcal{F}$ be the neighborhood filter of $0$, and let $\mathcal{B}=\left\{\R\right\}\cup\left\{\left[0,\frac{1}{n}\right)\right\}_{n\in\N}$. Then, $\mathcal{F}\subseteq \left[\mathcal{B}\right]$ but $\mathcal{F}\cap \mathcal{B}=\left\{\R\right\}$, hence $[\mathcal{F}\cap \mathcal{B}]=\{\R\}\ne \mathcal{F}=\mathcal{F}\cap \left[\mathcal{B}\right]$.

With every net $(x_\alpha)$ in $X$ we associate a filter on $X$, namely, its \term{tail filter}, $[x_\alpha]\defeq \{ F\subseteq X \mid F \text{ contains a tail of } (x_\alpha) \}$.  Conversely, for every filter $\mathcal{F}$ on $X$ there exists a (generally not unique) net $(x_\alpha)$ in $X$ so that $\mathcal{F}=[x_\alpha]$.  In particular, given a filter $\mathcal{F}$, define $A\defeq \{(F,x) \mid x\in F\in \mathcal{F}\}$ and order $A$ by reverse inclusion on the first component.  For every $\alpha=(F,x)\in A$ let $x_\alpha \defeq x$.  Then $(x_\alpha)$ is a net in $X$ and $[x_\alpha]=\mathcal{F}$. Hence we can (and do) work with either nets or filters, as is convenient.

A \term{convergence structure} on $X$ is defined by specifying, for every $x \in X$, those nets that converge to $x$, denoted $x_\alpha\to x$, subject to the following axioms.\begin{enumerate}

    \item Constant nets converge: if $x_\alpha=x$ for every $\alpha$ then $x_\alpha\to x$.

    \item If a net converges to $x$ then every quasi-subnet of it converges to $x$.

    \item Suppose that $(x_\alpha)$ and $(y_\alpha)$ both converge to $x$. Let $(z_\alpha)$ be a net in $X$ such that $z_\alpha\in\{x_\alpha,y_\alpha\}$ for every $\alpha$. Then $z_\alpha\to x$.

\end{enumerate}
The set $X$ equipped with a convergence structure is called a \term{convergence space}.

Although it is standard in the literature to define convergence structures in terms of filters, see for instance \cite{Beattie:02}, for applications in analysis it is often convenient to work with nets, and in \cite{OBrien:23} the above definition of a `net convergence structure' is shown to be equivalent to the usual one in terms of filters:  Given a convergence structure on $X$ as defined above, we declare a filter on $X$ to converge to $x$ if there exists a net $x_\alpha \to x$ so that $\mathcal{F}=[x_\alpha]$; this defines a filter convergence structure in the sense of \cite{Beattie:02}.  Conversely, given a filter convergence structure as in \cite{Beattie:02}, a net $(x_\alpha)$ in $X$ converges to $x\in X$ if the tail filter $[x_\alpha]$ converges to $x$; this yields a convergence structure as defined above.  The procedures for passing between net- to filter convergence structures are inverses of each other.

We remark that a slightly different definition of a net convergence structure was introduced in \cite{Aydin:21}, there referred to as a `convergence'. Every convergence structure is a convergence in the sense of \cite{Aydin:21}, but not conversely, as the following example shows, see also Remark \ref{Rem:  Strong order conv}.

\begin{ex}
On $\R^2$, declare $x_\alpha \goeseta x$ if there exists $\alpha_0\in A$ so that $\pi_1(x_\alpha)=\pi_1(x)$ for all $\alpha \ge \alpha_0$, or $\pi_2(x_\alpha)=\pi_2(x)$ for all $\alpha \ge \alpha_0$, with $\pi_1$ and $\pi_2$ denoting the coordinate projections.  It is easy to see that $\eta$ is a convergence in the sense of \cite{Aydin:21}.  However, it does not satisfy axiom (iii) in the definition of a convergence structure.  Indeed, let $x_n\defeq(1,0)$ and $y_n\defeq (0,1)$ for every $n\in\N$.  Let $z_n\defeq x_n$ if $n$ is even, and $z_n\defeq y_n$ if $n$ is odd.  Then $x_n\goeseta (0,0)$ and $y_n\goeseta (0,0)$, but $(z_n)$ is not $\eta$-convergent to $(0,0)$. \qed
\end{ex}

When simultaneously dealing with more than one convergence structure on $X$ we label them $\lambda$, $\eta$, etc., and write $x_\alpha\goesl x$ to denote that a net $(x_\alpha)$ converges to $x$ with respect to $\lambda$, and likewise for filters.  Given two convergence structures $\lambda$ and $\eta$ on $X$, we say that $\lambda$ is stronger than $\eta$ (or $\eta$ is weaker than $\lambda$) and write $\lambda\ge\eta$ if for every net $(x_\alpha)$ in $X$, $x_\alpha\goesl x$ implies that $x_\alpha\goeseta x$; equivalently, for every filter $\mathcal{F}$ on $X$, $\mathcal F\goesl x$ implies that $\mathcal F\goeseta x$.

Most basic concepts from topology, such as continuity of functions, generalise in the obvious way to the setting of convergence spaces:  a characterisation in terms of nets, sequences or filters becomes the definition.  Recall that if $f\colon X\to Y$ is a function and  $\mathcal F$ is a filter on $X$ then we write $f(\mathcal F)$ for the filter generated by $\bigl\{f(A)\mid A\in\mathcal F\bigr\}$. For a net $(x_\alpha)$ in $X$ we have $f\bigl([x_\alpha]\bigr)=\bigl[f(x_\alpha)\bigr]$. A function $f$ between two convergence spaces is \term{continuous} if $x_\alpha\to x$ implies $f(x_\alpha)\to f(x)$ for every net $(x_\alpha)$. Equivalently, if $\mathcal F\to x$ implies $f(\mathcal F)\to f(x)$ for every filter~$\mathcal F$. A convergence structure on a vector space is \term{linear} if the vector space operations are continuous; we call such a space a \term{convergence vector space}. We refer the reader to \cite{Beattie:02}, \cite{Bilokopytov:24} and \cite{OBrien:23} for any undeclared concepts.

One important point where the theory of convergence spaces diverges from topology concerns the closure of a set.  We briefly recall the main points. Let $X$ be a convergence space with convergence structure $\eta$ and $A$ a subset of $X$.  Recall that the \term{adherence} of $A$ is the set of limits of nets in $A$, and is $\overline{A}^{1,\eta}$ or just $\overline{A}^1$.  $A$ is \term{closed} if $A=\overline{A}^{1}$, and the \term{closure} of $A$ is the smallest closed subset $\overline{A}^\eta$ (or just $\overline{A}$) of $X$ containing $A$.  Clearly, $\overline{A}^1\subseteq \overline{A}$; in general the inclusion may be strict.  For an ordinal $\kappa$, we define $\overline{A}^{\kappa+1}$ as the adherence of $\overline{A}^\kappa$;  thus, in particular, $\overline{A}^{2}$ is the adherence of $\overline{A}^{1}$, etc.  If $\kappa$ is a limit ordinal then $\overline{A}^\kappa$ is defined as the union of $\overline{A}^\iota$ for all the ordinals $\iota<\kappa$. It is shown in \cite{Bilokopytov:24} that $\overline{A}^\kappa=\overline{A}$, for a large enough ordinal $\kappa$.

\subsection{Various countability properties of convergence spaces}\label{Sec:  Sequential}

One motivation for studying countability properties of convergence spaces is that under such conditions sequential arguments often suffice.  We briefly recall two relevant `sequential' properties of convergence spaces.

Let $X$ be a convergence space.  Every closed subset $B$ of $X$ is sequentially closed; that is, if $x_n\in B$ for every $n\in\N$ and $x_n\to x$ then $x\in B$. The converse is false even in topological spaces.  We call $X$ \term{sequential} if every sequentially closed set in $X$ is closed.  If the adherence of every set coincides with its sequential adherence, $X$ is called \term{Fr\'{e}chet-Urysohn}; that is, $X$ is Fr\'{e}chet-Urysohn if for every $x\in X$, the range of every net $(x_\alpha)$ converging to $x$ contains a sequence $(x_{\alpha_n})$ that converges to $x$.  It is easy to see that every Fr\'{e}chet-Urysohn convergence space is sequential.

\begin{prop}\label{conv}
Let $A\subseteq X$. If $X$ is sequential, then $\overline{A}=\overline{A}^{\omega_{1}}$; that is for every $x\in \overline{A}$ there is a countable successor ordinal $m$ such that $x\in \overline{A}^{m}$.
\end{prop}
\begin{proof}
Let $\left(x_{n}\right)$ be a convergent sequence in $\overline{A}^{\omega_{1}}$.  For every $n$ there is a countable ordinal $m_{n}$ such that $x_{n}\in \overline{A}^{m_{n}}$.  Let $m=\sup m_{n}$, which is a countable ordinal.  It follows that $\left\{x_{n}\right\}_{n\in\N}\subseteq \overline{A}^{m}$ so that the limit of $(x_n)$ is in $\overline{A}^{m+1}\subseteq \overline{A}^{\omega_{1}}$. Hence, $\overline{A}^{\omega_{1}}$ is sequentially closed in a sequential space, thus closed.
\end{proof}

\begin{rem}\label{conv1}
We note that the proof of Proposition \ref{conv} actually yields a slightly stronger result:  If $X$ is sequential, then every sequence in $\overline{A}$ is contained in $\overline{A}^{m}$, for some countable ordinal $m$.
\qed\end{rem}

The definition of a first countable topological space has been generalised to convergence spaces in two ways.  According to \cite[Definition 1.6.3]{Beattie:02}, a convergence space $X$ is \term{first countable} if $\mathcal F\to x$ implies that there exists a filter $\mathcal G\subseteq \mathcal{F}$ with a countable base such that $\mathcal G\to x$.  We say that $X$ is \term{strongly first countable} if for every $x$ there exists a countable collection $\mathcal H$ of subsets of $X$ such that if $\mathcal F\to x$ then there exists a filter base $\mathcal B$ with $\mathcal B\subseteq\mathcal F\cap\mathcal H$ and $[\mathcal B]\to x$, see \cite[Definition 1.6.5]{Beattie:02}.  Such an $\mathcal{H}$ is called a \term{countable local basis} of $X$ at $x$.  Clearly, every strongly first countable convergence structure is first countable.  In the case of a linear convergence structure on a vector space, it suffices to assume that $x=0$.  The definition of strong first countability can be improved as follows.

\begin{prop}\label{prop:  Existence of convergent directed fixed local basis}
Let $X$ be a convergence space.  Then $X$ is strongly first countable if and only if for every $x\in X$ there exists a countable filter base $\mathcal{H}_x$ with the following properties:
\begin{enumerate}[(i)]

    \item Every $H\in \mathcal{H}_x$ contains $x$;

    \item For every $G,H\in \mathcal{H}_x$ we have $G\cap H\in \mathcal{H}_x$;

    \item $[\mathcal{H}_x]\to x$;

    \item For a filter $\mathcal{F}$ we have that $\mathcal{F}\to x$ if and only if $[\mathcal{H}_x\cap \mathcal{F}]\to x$.

\end{enumerate}
\end{prop}
\begin{proof}
We only need to prove necessity. Let $X$ be a strongly first countable convergence space, $x\in X$, and $\mathcal{H}$ a countable local basis at $x$.  Let $\mathcal{H}_x \defeq \{H_{1}\cap...\cap H_{n} \mid x\in H_{1},...,H_{n}\in \mathcal{H}\}$. It is clear that the first two conditions are satisfied for  $\mathcal{H}_x$.

Let $\mathcal{F}$ be a filter on $X$. If $[\mathcal{H}_x\cap \mathcal{F}]\to x$, then $\mathcal{F}\to x$ due to the fact that $[\mathcal{H}_x\cap \mathcal{F}]\subseteq \mathcal{F}$. Conversely, if $\mathcal{F}\to x$, then $\mathcal{F}\cap [x]\to x$, and so there is a filter base $\mathcal B\subseteq [x]\cap\mathcal{F}\cap \mathcal H$ such that $[\mathcal B]\to x$. Then, $\mathcal B\subseteq \mathcal{H}_x\cap\mathcal{F}$, and so $[\mathcal{H}_x\cap\mathcal{F}]\to x$. In particular, taking $\mathcal{F}:= [x]$ yields $[\mathcal{H}_x]=[\mathcal{H}_x\cap [x]]\to x$.
\end{proof}

We will call a filter base with the properties listed in Proposition \ref{prop:  Existence of convergent directed fixed local basis} a \term{fixed intersection-closed local basis} at $x$. \medskip

The definition of first countability may be restated in terms of nets as follows:  $X$ is first countable if for every $x_\alpha \to x$ there exists a net $(y_\gamma)_{\gamma\in \Gamma}\to x$ such that $(x_\alpha)$ is a quasi-subnet of $(y_{\gamma})$ and $\Gamma$ contains a countable cofinal set, see \cite[Proposition 6.1]{OBrien:23}.  We call $X$ \term{weakly first countable} if for every $x_\alpha \to x$ in $X$ there exists an increasing sequence $(\alpha_n')$ in $A$ so that if $\alpha_n\ge \alpha_n'$ for every $n\in\N$ then $x_{\alpha_n}\to x$.  It was proven in \cite[Proposition 6.2]{OBrien:23} that first countability implies weak first countability.  Clearly every weakly first countable convergence space, in particular every first countable space, is  Fr\'{e}chet-Urysohn. For topological spaces, strong first countability and first countability both reduce to the usual definition of a first countable topological space.  More generally, a pretopological convergence space (see \cite[Definition 3.1.5]]{Beattie:02}) is first countable if and only if it is strongly first countable, see \cite[Remark 1.6.6]{Beattie:02}.  Let us show that, in this case, weak first countability is also equivalent to first countability.

\begin{prop}\label{Prop:  First countable char}
Let $X$ be a convergence space.  Consider the following statements. \begin{enumerate}

      \item[(i)] $X$ is strongly first countable.
      \item[(ii)] $X$ is first countable.
      \item[(iii)] $X$ is weakly first countable.
\end{enumerate}

Always, (i) implies (ii) and (ii) implies (iii).  If $X$ is pre-topological then all three statements are equivalent.
\end{prop}

\begin{proof}
As mentioned, (i) implies (ii), and (ii) implies (i) if $X$ is pretopological; that (ii) implies (iii) is proven in \cite[Proposition 6.2]{OBrien:23}.  Assume that $X$ is pretopological and (iii) holds.

Fix $x\in X$ and let $\mathcal{N}$ denote the neighbourhood filter at $x$.  Let $\Gamma \defeq \{(U,x) \mid x \in U\in \mathcal{N}\}$, ordered by reverse inclusion of the first component.  For $\gamma=(U,x)\in \Gamma$ let $x_\gamma=x$.  Then $(x_\gamma)$ is a net in $X$ and, since $X$ is pretopological, $[x_\gamma]=\mathcal{N}\to x$.  By assumption, there exists an increasing sequence $(\gamma_n')$ in $\Gamma$ so that if $\gamma_n\ge \gamma_n'$ for all $n\in\N$ then $x_{\gamma_n}\to x$.  For every $n\in\N$ let $\gamma_n'=(U_n,x_n)$.  Let $\mathcal{F}\defeq [\{U_n \mid n\in\N\}]\subseteq\mathcal{N}$.  Since $X$ is pretopological,
\[
\mathcal{F}=\bigcap\left\{ [x_{\gamma_n}] ~:~ \gamma_n \ge \gamma_n',~ n\in\N\right\} \to x.
\]
Since $\mathcal{N}$ is the least (w.r.t. inclusion) filter converging to $x$, we have $\mathcal{N}\subseteq \mathcal{F}$. Therefore $\mathcal{N}=\mathcal{F}$, hence $\mathcal{N}$ has a countable base so that $X$ is first countable.
\end{proof}

The property of being (strongly) first countable is stable under certain operations.  By \cite[Corollary 1.6.8]{Beattie:02}, a subspace of a (strongly) first countable convergence space is again (strongly) first countable. More generally, if $\{X_n\}_{n\in \N}$ is a sequence of (strongly) first countable spaces and for every $n\in \N$,  $f_n\colon X\to X_{n}$ is a map from some set $X$ to $X_n$, then the \term{initial convergence structure} on $X$ with respect to this sequence of maps (i.e. $x_{\alpha}\to x$ if $f_{n}(x_\alpha)\to f_n(x)$, for every $n\in\N$) is again (strongly) first countable; see \cite[Proposition 1.6.7(i)]{Beattie:02}. Furthermore, let $\{X_i\}_{i\in I}$ be a family of convergence spaces and, for every $i\in I$,  $f_i\colon X_i\to X$ a map from $X_i$ to some set $X$.  If the $X_i$ are all first countable, then the \term{final convergence structure} on $X$ with respect to the family of maps $\{f_i\}_{i\in I}$, that is, the strongest convergence structure on $X$ that makes all the $f_i$'s continuous, is again first countable.  If $I$ is countable, each $f_i$ is injective, and every $X_i$ is strongly first countable, then $X$ is strongly first countable; see \cite[Definition 1.2.9 and Proposition 1.6.7(iii)]{Beattie:02}.

For completeness, we establish analogous permanence results for weak first countability. Clearly, this property passes down to subspaces. We now consider initial and final convergence structures.

\begin{prop}
Let $X$ be a nonempty set.  \begin{enumerate}
      \item[(i)] Let $\{X_n\}_{n\in \N}$ be a sequence of weakly first countable spaces and for every $n\in \N$,  $f_n\colon X\to X_{n}$ a map from $X$ to $X_n$.  Then the initial convergence structure on $X$ with respect to this sequence of maps is weakly first countable.
      \item[(ii)] Let $\{X_i\}_{i\in I}$ be a family of weakly first countable convergence spaces, and $f_i$ a function from $X_i$ to $X$ for every $i\in I$.  Then the final convergence structure on $X$ with respect to the family of maps $\{f_i\}_{i\in I}$ is weakly first countable.
\end{enumerate}
\end{prop}

\begin{proof}
(i): Let $x_\alpha \xrightarrow{} x$ in the initial convergence structure on $X$ with respect to $\{f_n\}_{n\in \N}$. It follows that $f_{n}(x_\alpha)\to f_n(x)$, for every $n\in\N$. For every $n\in\N$ the assumption of weak first countability of $X_{n}$ yields a sequence $(\alpha_{m}^{n})_{m\in\N}$ such that whenever $\alpha_{m}\ge\alpha_{m}^{n}$, for every $m\in\N$, then $f_{n}\left(x_{\alpha_{m}}\right)\to f_{n}(x)$ in $X_{n}$. Find $\beta_{m}\ge \alpha_{k}^{n}$, for every $k,n=1,...,m$. Assume that $\alpha_{m}\ge\beta_{m}$, for every $m\in\N$. Then, for every $n\in\N$ we have $\alpha_{m}\ge\alpha_{m}^{n}$, for $m\ge n$, and so $f_{n}\left(x_{\alpha_{m}}\right)\to f_{n}(x)$ in $X_{n}$. We conclude that $x_{\alpha_{m}} \xrightarrow{} x$. \medskip

(ii): Let $x_\alpha \xrightarrow{} x$ in the final convergence structure on $X$ with respect to $\{f_i\}_{i\in I}$.  If $(x_\alpha)$ is eventually constant, then we are done.  Hence, suppose that this is not the case.  It follows from \cite[Definition 1.2.9]{Beattie:02} and the correspondence between nets and filters that there exists $m\in \N$ and $i_1,\ldots,i_m\in I$ so that for every $k=1,\ldots ,m$ there exists $x_k\in X_{i_k}$ and a net $x_\beta^k\xrightarrow{} x_k$ in $X_{i_k}$ so that $f_{i_k}(x_k)=x$ and for every $\beta_0\in B$ there exists $\alpha_0\in A$ so that $\{x_\alpha\}_{\alpha\ge \alpha_0}\subseteq \bigcup_{k=1}^m\{f_{i_k}(x_{\beta}^k)\}_{\beta\ge\beta_0}$.

By the weak first countability of the spaces $X_{i_1},\ldots, X_{i_m}$ there exists an increasing sequence $(\beta_n')$ in $B$ so that if $\beta_n \ge \beta_n'$ for every $n\in \N$, then $x_{\beta_n}^k\xrightarrow{} x_k$ in $X_{i_k}$ for every $k=1,\ldots, m$.  Select and increasing sequence $(\alpha_n')$ in $A$ so that $\{x_\alpha\}_{\alpha\ge \alpha_n'}\subseteq \bigcup_{k=1}^m\{f_{i_k}(x_{\beta}^k)\}_{\beta\ge\beta_n'}$ for every $n\in \N$.  Let $\alpha_n\ge \alpha_n'$ for every $n\in \N$.  For each $k=1,\ldots,m$, define a sequence $(z_n^k)$ in $X_{i_k}$ as follows:  If there exists $\beta \ge \beta_n'$ so that $x_{\alpha_n}=f_{i_k}(x_\beta^k)$, let $z_n^k\defeq x_\beta^k$; otherwise let $z_n^k\defeq x_{\beta_n'}^k$.  Then $z_n^k \xrightarrow{} x_k$ in $X_{i_k}$ for every $k=1,\ldots,m$.  Since $x_{\alpha_n}\in \{f_{i_k}(z_n^k) ~:~ k=1,\ldots,m\}$ for every $n\in \N$, it follows that $\{x_{\alpha_n}\}_{n\ge n_0} \subseteq \bigcup_{k=1}^m \{f_{i_k}(z_n^k)\}_{n\ge n_0}$ for each $n_0\in\N$.  Therefore $x_{\alpha_n}\xrightarrow{} x$ in the final convergence structure on $X$ with respect to $\{f_i\}_{i\in I}$.
\end{proof}

In conclusion, let us come back to the hierarchy of the countability properties. We call a convergence space \term{monotone Fr\'{e}chet-Urysohn} if for every $x_\alpha \to x$ there exists an increasing sequence $(\alpha_n)$ of indexes so that $x_{\alpha_n}\to x$. Clearly, this property implies the Fr\'{e}chet-Urysohn property. It is also easy to see that every weakly first countable convergence space is monotone Fr\'{e}chet-Urysohn.

\begin{question}
\begin{enumerate}
\item[(i)] Does there exist a Fr\'{e}chet-Urysohn convergence space which is not monotone Fr\'{e}chet-Urysohn?
\item[(ii)] Does there exist a monotone Fr\'{e}chet-Urysohn convergence space which is not weakly first countable?
\item[(iii)] Does there exist a weakly first countable convergence space which is not first countable?
\item[(iv)] Do there exist examples for parts (i) and (ii) in the class of topological spaces?
\end{enumerate}
\end{question}

\begin{rem}
After posting a preprint of the pre-review version of this paper the authors were contacted by Omer Cantor who was able to provide solutions to all parts of the question. These will be published elsewhere.
\qed\end{rem}

\subsection{Bornological convergence and the bounded modification}\label{Sect:Bornological}

If $(X,\eta)$ is a convergence space and $\mathcal{B}$ is a bornology on $X$, then the \term{bounded modification} of $\eta$ by $\mathcal{B}$ is defined as follows:  $x_\alpha \goesetaB x$ if $x_\alpha\goeseta x$ and a tail of $(x_\alpha)$ belongs to $\mathcal{B}$.  In terms of filters, $\mathcal{F}\goesetaB x$ if $\mathcal{F}\goeseta x$ and $\mathcal{F}\cap \mathcal{B}\neq \varnothing$.  This is the strongest convergence on $X$ that agrees with $\eta$ on sets in $\mathcal{B}$.  We call $\eta$ \term{locally $\mathcal{B}$-bounded} if $\eta=\eta_{\mathcal{B}}$. Note that if $\eta$ is locally $\mathcal{B}$-bounded, and $\mathcal{H}$ is a local basis for $\eta$ at $x\in X$, then $\mathcal{B}\cap \mathcal{H}$ is also a local basis at $x$.

The following characterisation of $\eta_\mathcal{B}$-closed sets is used below.

\begin{prop}\label{Prop:  Closed sets wrt bounded mod}
Let $(X,\eta)$ be a convergence space, $\mathcal{B}$ a bornology on $X$ with base $\mathcal{D}$, and $A\subseteq X$.  Then $A$ is $\eta_\mathcal{B}$-closed if and only if $A\cap D$ is $\eta$-closed in $D$ for every $D\in \mathcal{D}$.
\end{prop}

\begin{proof}
Assume that $A\cap D$ is $\eta$-closed in $D$ for every $D\in \mathcal{D}$.  Let $(x_\alpha )$ be a net in $A$ so that $x_\alpha \goesetaB x$.  Then there exists $\alpha_0\in A$ so that $\{x_\alpha\}_{\alpha\geq \alpha_0}\in \mathcal{B}$.  Therefore, there exists $D\in\mathcal{D}$ so that $\{x\}\cup \{x_\alpha\}_{\alpha\geq \alpha_0}\subseteq D$. By assumption, $A\cap D$ is $\eta$-closed in $D$.  Since $x_\alpha \goeseta x\in D$ it therefore follows that $x\in A$, hence $A$ is $\eta_\mathcal{B}$-closed.

The converse follows immediately from the fact that $\eta$ and $\eta_\mathcal{B}$ agree on members of $\mathcal{B}$.
\end{proof}

\begin{prop}\label{Prop: Bdmod-firs-ctb}
Let $(X,\eta)$ be a convergence space and $\mathcal{B}$ a bornology on $X$.  The following statements are true. \begin{enumerate}[(i)]

    \item\label{Prop: Bdmod-firs-ctb 1st c} If $\eta$ is first countable then so is $\eta_\mathcal{B}$.

    \item\label{Prop: Bdmod-weak-firs-ctb 1st c} If $\eta$ is weakly first countable then so is $\eta_\mathcal{B}$.

    \item\label{Prop: Bdmod-Frechet-Urosohn-Frechet-Urysohn} If $\eta$ is Fr\'{e}chet-Urysohn then so is $\eta_\mathcal{B}$.

    \item\label{Prop: Bdmod-Sequential-Sequential} If $\eta$ is sequential and $\mathcal{B}$ has a base consisting of $\eta$-closed sets, then $\eta_\mathcal{B}$ is sequential.

    \item\label{Prop: Bdmod-firs-ctb str 1st c} If $\eta$ is strongly first countable and $\mathcal{B}$ has a countable base, then $\eta_\mathcal{B}$ is strongly first countable.
\end{enumerate}
\end{prop}

\begin{proof}
(i):~ Assume that $\eta$ is first countable, and let $\mathcal{F}\goesetaB 0$.  Then $\mathcal{F}\goeseta 0$ and there exists a set $B\in\mathcal{B}\cap \mathcal{F}$.  By assumption, there exists a filter $\mathcal{G}\subseteq \mathcal{F}$ with a countable base $\{G_n\}_{n\in\N}$ so that $\mathcal{G}\goeseta 0$.  Let $\mathcal{D}=[\{B\cap G_n \mid n\in\N\}]$.  Then $\mathcal{G}\subseteq \mathcal{D}\subseteq \mathcal{F}$ and $B\in\mathcal{D}\cap\mathcal{B}$ so that, in particular, $\mathcal{D}\goesetaB 0$.  Hence $\eta_\mathcal{B}$ is first countable.

(ii):~ Assume that $\eta$ is weakly first countable.  Let $x_\alpha \goesetaB 0$.  Then $x_\alpha \goeseta 0$, and there exists $\alpha'_0\in A$ so that $\{x_\alpha\}_{\alpha \geq \alpha'_0}\in \mathcal{B}$.  Since $(x_\alpha)_{\alpha \geq \alpha'_0}\goeseta 0$, there exists an increasing sequence $(\alpha'_n)$ in $A$ so that $\alpha'_1 \geq \alpha'_0$ for every $n\in \N$, and, if $\alpha_n\geq \alpha_n'$ for every $n\in \N$ then $x_{\alpha_n}\goeseta 0$. For any such sequence $(\alpha_n)$ we also have  $\{x_{\alpha_n}\}_{n\in \N}\subseteq \{x_{\alpha}\}_{\alpha \geq \alpha'_0}\in \mathcal{B}$ so that $\{x_{\alpha_n}\}_{n\in \N}\in \mathcal{B}$. Therefore $x_{\alpha_n}\goesetaB 0$ whenever $\alpha_n \geq \alpha_n'$ for all $n\in \N$.

(iii):~ A similar argument as in (ii) shows that $\eta_\mathcal{B}$ is Fr\'{e}chet-Urysohn whenever $\eta$ is.

(iv):~ Let $\eta$ be sequential and assume that $\mathcal{B}$ has a base $\mathcal{D}$ consisting of $\eta$-closed sets.  Suppose that $A\subseteq X$ is sequentially $\eta_\mathcal{B}$-closed.  Let $D \in\mathcal{D}$.  Since $\eta$ and $\eta_\mathcal{B}$ agree on members of $\mathcal{B}$, it follows that $A\cap D$ is sequentially $\eta$-closed in $D$.  But $D$ is $\eta$-closed, so $A\cap D$ is sequentially $\eta$-closed, hence $\eta$-closed, in $X$.  Consequently, $A\cap D$ is $\eta$-closed in $D$ so that $A$ is $\eta_\mathcal{B}$-closed by Proposition \ref{Prop:  Closed sets wrt bounded mod}.

(v):~ Suppose that $\eta$ is strongly first countable and $\mathcal{B}$ has a countable base, $\{B_n\}_{n\in\N}$.  It is easy to see that $\eta_\mathcal{B}$ is the final convergence structure on $X$ with respect to the inclusions $j_n:B_n\to X$, $n\in\N$.  Therefore $\eta_\mathcal{B}$ is strongly first countable.
\end{proof}

We now assume that $X$ is a vector space, that $\eta$ is a linear convergence structure on $X$, and that $\mathcal{B}$ is a linear bornology. In this case $\eta_{\mathcal{B}}$ is a linear convergence.  Denote by $\mathcal{N}_0$ the usual base of zero neighbourhoods in the scalar field.  A subset $B$ of $X$ is \term{bounded} if $[\mathcal{N}_0 B]\to 0$.  The collection of all bounded sets in $X$ is a vector bornology on $X$. We call $X$ \term{locally bounded} if it is locally $\mathcal{B}$-bounded, where $\mathcal{B}$ is the collection of all $\eta$-bounded sets. This means that every convergent net $x_\alpha \to 0$ has a bounded tail; equivalently, if every filter $\mathcal{F}\to 0$ contains a bounded set.

\begin{prop}\label{muB-fc0}
Let $\eta$ be a locally bounded linear convergence structure on a vector space $X$ and let $\mathcal{B}$ be the bornology of all $\eta$-bounded sets.  Then $\eta$ is strongly first countable if and only if there exists a (increasing) sequence $(B_n)$ in $\mathcal{B}$ such that $\left\{mB_{n}\right\}_{m,n\in\N}$ is a base of $\mathcal{B}$, so in particular $X=\mathrm{span}\bigcup B_{n}$, and the restriction of $\eta$ to $\mathrm{span}B_{n}$ is strongly first countable for every $n\in \N$.
\end{prop}

\begin{proof}
Suppose that $\eta$ is strongly first countable and let $\mathcal H$ be a countable, intersection-closed local basis for $\eta$ at $0$. Enumerate $\mathcal{B}\cap\mathcal H$ as $\{B_n\}_{n\in\N}$.  Replacing $B_{n}$ with $B_{1}+...+B_{n}$ if needed we may assume that $(B_n)$ is increasing.

We claim that $\left\{mB_{n}\right\}_{m,n\in\N}$ is a base for $\mathcal{B}$.  Fix $B\in\mathcal{B}$.  Since $\left[\mathcal{N}_{0}B\right]\goeseta 0$, $\left[\left[\mathcal{N}_{0}B\right]\cap\mathcal{H}\right]\goeseta 0$.  As $\eta$ is locally bounded there exists $A\in \left[\mathcal{N}_{0}B\right]\cap\mathcal{H}\cap \mathcal{B}$.  The fact that $A\in \mathcal{H}\cap \mathcal{B}$ implies that there exists $n\in\N$ so that $A\subseteq B_{n}$, and since $A\in \left[\mathcal{N}_{0}B\right]$, there exists $m\in\N$ so that $\frac{1}{m}B\subseteq A$.  We conclude that $B\subseteq mB_{n}$, so $\left\{mB_{n}\right\}_{m,n\in\N}$ is a base for $\mathcal{B}$.

Since a subspace of a strongly first countable convergence space is strongly first countable, it follows that the restriction of $\eta$ to $\mathrm{span}B_{n}$ is strongly first countable for every $n\in\N$.\medskip

We now prove the converse.  For every $n$, let $j_n\colon \mathrm{span}B_{n}\to X$ be the inclusion map; it is clearly $(\eta_{|\mathrm{span}B_{n}})$-to-$\eta$ continuous.  To complete the proof, it suffices to show that $\eta$ is the final convergence structure with respect to $\{j_n\}_{n\in\N}$.  To see that this is so, let $\lambda$ be a convergence structure on $X$ such that each $j_n$ is $(\eta_{|\mathrm{span}B_{n}})$-to-$\lambda$ continuous.  Let $x_\alpha\goeseta x$. Since $\eta$ is locally bounded we may assume, after passing to a tail of $(x_\alpha)$ if necessary, that $(x_\alpha)$ is bounded.  It follows that there exists $n\in\mathbb N$ such that $(x_\alpha)$ and $x$ are contained in $\mathrm{span}B_{n}$.  This yields $x_\alpha\xrightarrow{\eta_{|\mathrm{span}B_{n}}}x$ and, therefore, $x_\alpha\goesl x$. We conclude that $\eta\ge\lambda$; this proves the claim.
\end{proof}

Let $X$ be a vector space and $\mathcal{B}$ a vector bornology on $X$.  We define the \term{bornological convergence} structure induced by $\mathcal B$ on $X$ as follows:  $x_\alpha\goesmuB 0$ if there exists (a circled) set $B\in\mathcal B$ such that for every $\varepsilon>0$ the set $\varepsilon B$ contains a tail of the net $(x_\alpha)$; we define $x_\alpha\goesmuB x$ if $x_\alpha-x\goesmuB 0$.  In terms of filters, $\mathcal F\goesmuB 0$ if $\mathcal F\supseteq[\mathcal N_0 B]$ for some $B\in\mathcal B$.  Clearly, the $\mu_\mathcal{B}$-bounded subsets of $X$ are precisely the members of $\mathcal{B}$; therefore $\mu_{\mathcal{B}}$ is locally bounded.

\begin{prop}\label{muB-fc}
Let $\mathcal B$ be a linear bornology on a vector space $X$. Then $\mu_{\mathcal B}$ is first countable. It is strongly first countable if and only if $\mathcal B$ has a countable base.
\end{prop}

\begin{proof}
Suppose $\mathcal F\goesmuB 0$. Then there exists a circled set $A\in\mathcal B$ such that $\frac1n A\in\mathcal F$ for all $n\in\N$.  It follows that the filter generated by $\{\frac1n A\mid n\in\mathbb N\}$ is contained in $\mathcal F$ and $\mu_{\mathcal B}$-converges to zero.  Hence $\mu_{\mathcal B}$ is first countable.

If $\mu_{\mathcal B}$ is strongly first countable, then $\mathcal{B}$ has a countable base by Proposition \ref{muB-fc0}.  Suppose that $\mathcal B$ has a countable base, $\{B_n\mid n\in\mathbb N\}$. For every $n\in \N$ we define a convergence structure $\lambda_n$ on $\Span B_n$ as follows: $x_\alpha\xrightarrow{\lambda_n}x$ if for every $\varepsilon>0$ there exists $\alpha_0$ such that $x_\alpha-x\in\varepsilon B_n$ for all $\alpha\ge\alpha_0$.  Observe that $\lambda_n$ is strongly first countable for every $n$:  For every $x\in X$, the collection $\{x+\frac1m B_n\mid m\in\mathbb N\}$ is a local base for $\lambda_N$ at $x$.

A variation on the argument for the corresponding claim in the proof of Proposition \ref{muB-fc0} shows that $\mu_{\mathcal B}$ is the final convergence structure for the family of formal inclusions $j_n\colon(\Span B_n,\lambda_n)\to(X,\mu_{\mathcal B})$.  It follows that $(X,\mu_{\mathcal B})$ is strongly first countable.
\end{proof}

We refer the reader to \cite[Sections 3.7 \& 3.8]{Beattie:02} and \cite[Sections 6, 7 \& 8]{Bilokopytov:24} for more details on bornologies, bounded sets in convergence vector spaces, and the bounded modification of a convergence vector space.

\section{Locally solid convergence structures}\label{Sec:  Locally solid}

We now turn to (strong) first countability of general locally solid convergence spaces.  Throughout this section $X$ denotes an Archimedean vector lattice.  Recall that a subset $A$ of $X$ is \term{solid} if for all $x\in A$ and $y\in X$, if $|y|\le |x|$ then $y\in A$.  The \term{solid hull} of $A$ is the smallest solid subset $\Sol(A)$ of $A$ containing $A$.  The smallest solid linear subspace containing $A$ is called the \term{ideal} generated by $A$, and denoted $I_A$; if $u\in X$ we write $I_u$ or the ideal generated by $\{u\}$.  For a set $A\subseteq X$ we denote by $A^\wedge$ the set consisting of infima of finite subsets of $X$; likewise, $A^\vee$ consist of suprema of finite subsets of $A$.  We refer the reader to \cite{Aliprantis:03,Aliprantis:06,Luxemburg:71} for undeclared notation and terminology.

A \term{locally solid} convergence structure on $X$ is a linear convergence structure so that if $x_\alpha \to 0$ and $|y_\alpha|\le |x_\alpha|$ for all $\alpha$, then $y_\alpha \to 0$.  If $X$ is equipped with a locally solid convergence structure $\eta$, we call the pair $(X,\eta)$, or just $X$, a locally solid convergence space.  The following result from \cite{Bilokopytov:24} gives some useful characterisations of a locally solid convergence space.

\begin{thm}\label{def-ls}
Let $\lambda$ be a linear convergence structure on $X$.  The following are equivalent. \begin{enumerate}[(i)]

    \item\label{def-ls-ls} $\lambda$ is locally solid;

    \item\label{def-ls-fil} If $\mathcal F\to 0$ then there is a filter $\mathcal G$ with a base of solid sets such that $\mathcal G\subseteq\mathcal F$ and $\mathcal G\to 0$;

    \item\label{def-ls-fil-hull} If $\mathcal F\to 0$ then $\Sol(\mathcal F)\defeq [\{\Sol(F) \mid F\in\mathcal{F}\}]\to 0$;

    \item\label{def-ls-net-dom} If $y_\gamma\to 0$ and for every $\gamma_0$ there exists $\alpha_0$ such that for every $\alpha\ge\alpha_0$ there exists $\gamma\ge\gamma_0$ with $\abs{x_\alpha}\le\abs{y_\gamma}$, then $x_\alpha\to 0$.
  \end{enumerate}
\end{thm}

In \cite{Aydin:21}, the concept of a full lattice convergence was introduced, using their definition of a convergence:  A linear convergence structure $\eta$ on a vector lattice $X$ is \term{locally full} if $0\le y_\alpha \le x_\alpha$ for every $\alpha$ and $x_\alpha \goeseta 0$ imply $y_\alpha \goeseta 0$.  As is shown in \cite[Proposition 3.3]{Bilokopytov:24}, a linear convergence structure is locally solid if and only if it is locally full and the modulus operation is continuous.  Therefore, a linear convergence structure on $X$ is locally solid if and only if it is a full lattice convergence as defined in \cite[Definition 2.1]{Aydin:21}.

\begin{rem}\label{Rem: Locally full}
We note that a linear convergence structure $\eta$ is locally full if and only if the following holds:  For all nets $(x_\alpha)$ and $y_\gamma \goeseta 0$ in $X_+$, if for every $\gamma_0$ there exists $\alpha_0$ such that for every $\alpha\ge\alpha_0$ there exists $\gamma\ge\gamma_0$ with $x_\alpha\le y_\gamma$, then $x_\alpha\to 0$.  This follows from a similar argument to that in the proof of \cite[Theorem 3.1]{Bilokopytov:24}. \qed
\end{rem}

Note that if $\eta$ is a locally solid convergence structure on $X$ then, as a consequence of Theorem \ref{def-ls},  if $\mathcal{H}$ is a local basis for $\eta$ at $0$, then $\mathcal{H}\cap \mathcal{S}$ is also a local basis, where $\mathcal{S}$ is the collection of all solid sets in $X$.  We may therefore always choose a local basis for $\eta$ at $0$ which consists of solid sets.\medskip

We call a locally solid convergence structure on $X$ \term{locally order bounded} if every convergent net has an order bounded tail; equivalently, every convergent filter contains an order bounded set.  For such a convergence structure, every bounded set is in fact order bounded.  Using a filter argument and the continuity of scalar multiplication, it is easy to see that converse is also true.  Hence, in a locally order bounded locally solid convergence space, the bounded sets are precisely the order bounded sets.  As every order bounded set is contained in an interval of the form $\left[-u,u\right]$ for some $u\ge0$ and $\mathrm{span}\left[-u,u\right]=I_{u}$, we arrive at the following consequence of Proposition \ref{muB-fc0}.

\begin{cor}\label{first-c-lemma}
Let $(X,\eta)$ be a locally solid, locally order bounded convergence space.  Then $X$ is strongly first countable if and only if there exists an increasing sequence $(v_n)$ in $X_+$ such that $X=I_{\{v_n\}}$ and the restriction of $\eta$ to $I_{v_n}$ is strongly first countable for every $n\in\N$.
\end{cor}

\begin{rem}\label{rem-strong}
Suppose that $X$ admits a completely metrizable locally solid topology.  Then there exists an increasing sequence $(v_n)$ in $X_+$ so that $X=I_{\{v_n\}}=\bigcup_{i=1}^\infty I_{v_n}$ if and only if $X$ has a strong unit.  Indeed, by appropriately scaling the sequence $(v_n)$ if necessary, we may assume that $\sum v_n$ is convergent with respect to the metric topology.  Then $e\defeq \sum v_n$ is a strong unit since $v_n\le e$ for every $n\in \N$.  The converse is trivial. \qed
\end{rem}

With any locally solid convergence structure $\eta$ on $X$ we associate its \term{unbounded modification} $\mathrm{u}\eta$ as follows:  $x_\alpha \goesueta x$ if $y\wedge|x-x_\alpha| \goeseta 0$ for all $y\ge 0$.  More generally, if $I$ is an ideal in $X$ we can unbound $\eta$ with respect to $I$:  $x_\alpha \goesuIe x$ if $y\wedge|x-x_\alpha| \goeseta 0$ for all $0\le y\in I$.  For any ideal $I$, the unbound modification ${\rm u}_I\eta$ is a locally solid convergence structure on $X$, see \cite[Section 4]{Bilokopytov:24}.  Next, we present a condition that guarantees that (strong) first countability passes to the unbounded modification.

\begin{prop}\label{fsc-unbdd}
Let $\eta$ be a locally solid convergence on $X$, $I$ an ideal in $X$, and $(v_n)$ a sequence in $I_+$ such that $I=I_{\{v_n\}}$.  \begin{enumerate}[(i)]

    \item If $\eta$ is first countable or strongly first countable then so is ${\rm u}_I\eta$.

    \item If $\eta$ is weakly first countable, then so is ${\rm u}_I\eta$.


\end{enumerate}

\end{prop}

\begin{proof}
(i):~ Suppose first that $\eta$ is first countable, and let
  \begin{math}
    \mathcal F\goesuIe 0.
  \end{math}
We need to show that there exists a filter $\mathcal G$ with a countable base such that $\mathcal G\subseteq\mathcal F$ and
  \begin{math}
    \mathcal G\goesuIe 0.
  \end{math}
Since $\mathrm{u}_I\eta$ is locally solid we may assume by Theorem~\ref{def-ls}\eqref{def-ls-fil-hull}, replacing $\mathcal F$ with $\Sol(\mathcal F)$ if needed, that $\mathcal F$ has a base of solid sets. Let $n\in\mathbb N$. Since $\abs{\mathcal F}\wedge v_n\goeseta 0$ by the definition of ${\rm u}_I\eta$, there exists a countable filter base $\mathcal B_n\subseteq\abs{\mathcal F}\wedge v_n$ such that $[\mathcal B_n]\goeseta 0$.  Enumerate $\mathcal B_n$ as $\{B_{n,k}\mid k\in\mathbb N\}$. For every $k$, it follows from $B_{n,k}\in\abs{\mathcal F}\wedge v_n$ that there exists a solid set $A_{n,k}\in\mathcal F$ such that $\abs{A_{n,k}}\wedge v_n\subseteq B_{n,k}$.  Consider the set $\mathcal A\subseteq\mathcal F$ of all finite intersections of sets in $\{A_{n,k}\mid n,k\in\mathbb N\}$. Clearly, $\mathcal A$ is countable, consists of solid sets, is contained in $\mathcal F$, and is closed under finite intersections. Hence, it is a filter base.  Let $\mathcal G \defeq [\mathcal A]$. Then $\mathcal G\subseteq\mathcal F$.

For every $n\in\N$, $[\mathcal B_n]\subseteq\abs{\mathcal G}\wedge v_n$ so that $\abs{\mathcal G}\wedge v_n\goeseta 0$. By \cite[Remark 4.6]{Bilokopytov:24}, we conclude that
  \begin{math}
    \mathcal G\goesuIe 0.
  \end{math}
Hence ${\rm u}_I\eta$ is first countable.\medskip

Assume that $\eta$ is strongly first countable. Let $\mathcal H$ be a countable intersection-closed local basis for $\eta$ at $0$.  By the comment after Theorem \ref{def-ls} we may assume that $\mathcal H$ consists of solid sets.  For $A\in\mathcal H$ and $k\in\mathbb N$, define
  \begin{math}
\widetilde{A}^k\defeq \bigl\{x\in X\mid \abs{x}\wedge v_k\in A\bigr\}.
  \end{math}
Let $\widetilde{\mathcal H}_0\defeq \{\widetilde{A}^k \mid A\in\mathcal{H}, k\in\N\}$, and let $\widetilde{\mathcal H}$ be the set of all finite intersections of sets in $\widetilde{H}_0$.  Clearly, $\widetilde{\mathcal H}$ is a countable collection of solid sets.  We show that it is a local basis for $\mathrm{u}_I\eta$.

Let $\mathcal F\goesuIe 0$. Fix $k\in\mathbb N$ and define
  \begin{math}
    \mathcal B_k\defeq \mathcal H\cap
     \bigl(\abs{\mathcal F}\wedge v_k\bigr).
  \end{math}
Since $\abs{\mathcal F}\wedge v_k\goeseta 0$, also $[\mathcal B_k]\goeseta 0$. We claim that $\widetilde{B}^k\in\mathcal F$ for every $B\in\mathcal B_k$.  Indeed, it follows from $B\in\abs{\mathcal F}\wedge v_k$ that there exists $A\in\mathcal F$ such that $\abs{A}\wedge v_k\subseteq B$; this yields $A\subseteq\widetilde{B}^k$ and, therefore, $\widetilde{B}^k\in\mathcal F$.

Let $\mathcal C\defeq \bigl\{\widetilde{B}^k\mid k\in\mathbb N, B\in\mathcal B_k\bigr\}$. It follows from $\mathcal C\subseteq\mathcal F$ that finite intersections of members of $\mathcal C$ are non-empty. Hence, the family $\mathcal D$ of all such intersections is a filter base and $\mathcal D\subseteq\mathcal F$. For every $k\in\mathbb N$ and $B\in\mathcal B_k$, we have $B\in\mathcal H$, hence $\widetilde{B}^k\in\widetilde{\mathcal H}_0$. It follows that $\mathcal D\subseteq\widetilde{\mathcal H}$. It remains to show that
  \begin{math}
    [\mathcal D]\goesuIe 0.
  \end{math}
As in the proof for first countability, it suffices to prove that $\bigabs{[\mathcal D]}\wedge v_k\goeseta 0$ for every $k\in N$.  This follows from the fact that $[\mathcal B_k]\goeseta 0$ and $[\mathcal B_k]\subseteq\bigabs{[\mathcal D]}\wedge v_k$.  Indeed, for every $B\in\mathcal B_k$, we have $\widetilde{B}^k\in\mathcal D$ and
  \begin{math}
    \abs{\widetilde{B}^k}\wedge v_k\subseteq B,
  \end{math}
so that $B\in\bigabs{[\mathcal D]}\wedge v_k$.

(ii):~ Suppose that $\eta$ is weakly first countable.  Let $x_\alpha \goesuIe 0$ so that $|x_\alpha|\wedge v_k \goeseta 0$ for every $k\in\N$.  For every $k\in \N$ there exists an increasing sequence $(\alpha_n^k)$ in $A$ so that if $\alpha_n\geq \alpha_n^k$ for every $n\in \N$, then $|x_{\alpha_n}|\wedge v_k \goeseta 0$.  Let $(\alpha_n')$ be an increasing sequence in $A$ so that for every $n\in \N$ and $k\leq n$, $\alpha_n^k \leq \alpha_n'$.  Consider a sequence $(\alpha_n)$ in $A$ so that $\alpha_n'\leq \alpha_n$ for all $n\in \N$.  For $k \in \N$, we have $\alpha_n^k\leq \alpha_n' \leq \alpha_n$ whenever $n\geq k$, so that $|x_{\alpha_n}|\wedge v_k \goeseta 0$.  Therefore, by \cite[Remark 4.6]{Bilokopytov:24}, $x_{\alpha_n}\goesuIe 0$; hence, ${\rm u}_I \eta$ is weakly first countable.
\end{proof}

\begin{rem}
If $\eta$ is Frech\'et-Urysohn and $I=I_{v}$ for some $v\in X_{+}$, then $x_\alpha \goesuIe x$ iff $v\wedge|x-x_\alpha| \goeseta 0$. It is then straightforward to show that ${\rm u}_I\eta$ is Frech\'et-Urysohn.

Assume that $\eta$ is idempotent, see \cite{Bilokopytov:23}, \cite{Bilokopytov:24}. Then, according to \cite[Proposition 3.1(v)]{Bilokopytov:23}, ${\rm u}_I\eta={\rm u}_{\overline{I}}\eta$.  In this case, if $\eta$ is Frech\'et-Urysohn and $I=\overline{I_{v}}$ for some $v\in X_+$, then ${\rm u}_I\eta$ is Frech\'et-Urysohn.

Moreover, under the assumption that $\eta$ is idempotent, Proposition \ref{fsc-unbdd} remains valid under the weaker assumption that $I=\overline{I_{\{v_n\}}}$.
\qed\end{rem}

The following result is a partial converse of Proposition \ref{fsc-unbdd}.

\begin{prop}\label{Porp: Ueta 1st countable implies eta 1st countable}
Let $\eta$ be a locally order bounded and locally solid convergence structure on $X$.  \begin{enumerate}[(i)]

    \item\label{first-ctb} If ${\rm u}\eta$ is first countable, then so is $\eta$.
    \item\label{Weakly-first-ctb} If ${\rm u}\eta$ is weakly first countable, then so is $\eta$.
    \item\label{Frechet-Urysohn} If ${\rm u}\eta$ is Fr\'{e}chet-Urysohn, then so is $\eta$.
    \item\label{Sequential} If ${\rm u}\eta$ is sequential and Hausdorff, then $\eta$ is sequential.
    \item\label{str-first-ctb} If ${\rm u}\eta$ is strongly first countable and there exists an increasing sequence $(v_n)$ in $X_+$ so that $X=I_{\{v_n\}}$, then $\eta$ is strongly first countable.

\end{enumerate}
\end{prop}

\begin{proof}
Let $\mathcal{B}$ be the bornology of all order bounded sets in $X$.  According to \cite[Proposition 8.28]{Bilokopytov:24}, $({\rm u}\eta)_{\mathcal{B}}=\eta$.  Therefore the result follows immediately from Proposition \ref{Prop: Bdmod-firs-ctb}.  For (\ref{Sequential}), it follows from \cite[Proposition 3.4]{Bilokopytov:24} and the assumption that ${\rm u}\eta$ is Hausdorff that $\mathcal{B}$ has a base of ${\rm u}\eta$-closed sets.
\end{proof}

Let $\eta$ be a linear convergence structure on $X$.  The \term{locally solid modification} $\mathrm{s}\eta$ of $\eta$ is defined as follows:  We declare $\mathcal{F}\goesseta 0$ if there exists $\mathcal{G}\goeseta 0$ with $X_+\in \mathcal{G}$ so that $\Sol(\mathcal{G})\subseteq \mathcal{F}$; in general, $\mathcal{F}\goesseta x$ if $\mathcal{F}-x\goesseta 0$.  In \cite[Definition 2.2]{Aydin:21} the so-called `fullification' of $\eta$ is defined, and in Definition 4.2 of the same paper, the authors introduce `absolute $\eta$ convergence\footnote{This is done in the context of their slightly different notion of a `convergence'.}.  The following result encompasses and refines \cite[Theorems 1 \& 5]{Aydin:21}.

\begin{prop}\label{Prop:  Locally solid mod properties}
Let $\eta$ be a linear convergence structure on $X$.  \begin{enumerate}[(i)]

    \item Let $(x_\alpha)$ in $X$ be a net in $X$ and $x\in X$. Then $x_\alpha \goesseta x$ if and only if there exists a net $0\leq y_\beta \goeseta 0$ so that for every $\beta_0\in B$ there exists $\alpha_0\in A$ so that for every $\alpha \ge \alpha_0$ there exists $\beta\ge \beta_0$ so that $|x-x_\alpha|\le y_\beta$.

    \item $\mathrm{s}\eta$ is the strongest locally solid convergence structure on $X$ so that $0\le x_\alpha \goeseta 0$ implies $x_\alpha \xrightarrow{} 0$.

    \item Assume that the absolute value operation is $\eta$-to-$\eta$ continuous at $0$.  Then $\mathrm{s}\eta$ is the strongest locally solid convergence structure which is weaker than $\eta$.

    \item Assume that $\eta$ is locally full.  For a net $(x_\alpha)$ in $X$ and $x\in X$, $x_\alpha \goesseta x$ if and only if $|x-x_\alpha|\goeseta 0$.  If $\eta$ is Hausdorff, then so is $\mathrm{s}\eta$.
\end{enumerate}
\end{prop}

\begin{proof}
(i):~ Replacing $(x_\alpha)$ with $(x_\alpha-x)$ if necessary, we may suppose that $x=0$.  Assume that $x_\alpha \goesseta 0$.  There exists a filter base $\mathcal{B}$ consisting of subsets of $X_+$ so that $[\mathcal{B}]\goeseta 0$ and $[\Sol(\mathcal{B})]\subseteq [x_\alpha]$.  Using the correspondence between filters and nets we find a net $(y_\beta)$ in $X_+$ so that $[y_\beta]=[\mathcal{B}]\goeseta 0$, hence $y_\beta\goeseta 0$.  For every $\beta_0\in B$ there exists $\alpha_0\in A$ so that $\{x_\alpha\}_{\alpha\ge \alpha_0}\subseteq \Sol(\{y_\beta\}_{\beta\ge\beta_0})$.  Let $\alpha \ge \alpha_0$.  There exists $\beta\ge \beta_0$ so that $|x_\alpha|\le |y_\beta|=y_\beta$.

Conversely, suppose that there exists a net $0\leq y_\beta \goeseta 0$ so that for every $\beta_0\in B$ there exists $\alpha_0\in A$ so that for every $\alpha \ge \alpha_0$ there exists $\beta\ge \beta_0$ so that $|x_\alpha|\le y_\beta$.  Then $[y_\beta]\goeseta 0$ and contains $X_+$.  By assumption, for every $\beta_0\in B$ there exists $\alpha_0 \in A$ so that $\{x_\alpha\}_{\alpha\ge \alpha_0}\subseteq \Sol(\{y_\beta\}_{\beta \ge \beta_0})$.  Therefore $\Sol([y_\beta])\subseteq  [x_\alpha]$.\medskip

(ii):~ Using \cite[Proposition 3.2.3]{Beattie:02} and Theorem \ref{def-ls} (\ref{def-ls-fil}), it is easy to see that $\mathrm{s}\eta$ defines a locally solid convergence structure on $X$.  We remark that condition (iii) in \cite[Proposition 3.2.3]{Beattie:02} follows from the fact that $\Sol(A+B)=\Sol(A)+\Sol(B)$ for all $A,B\subseteq X_+$.

It follows immediately from (i) that $0\le x_\alpha \goeseta 0$ implies $0\le x_\alpha \goesseta 0$. Assume that $\to$ is a locally solid convergence with this property and that $x_\alpha \goesseta 0$. Let $(y_{\beta})$ be the net from (i). Since $0\le y_\beta \goeseta 0$, it follows that $y_\beta\to0$. Due to local solidity of $\to$ and Theorem \ref{def-ls} we conclude that $x_\alpha \to 0$. Thus, $\mathrm{s}\eta$ is stronger than $\to$.\medskip

(iii):  In the light of (ii) it is enough to show that $\mathrm{s}\eta\le\eta$. Assume that $x_\alpha \goeseta 0$, and so by our assumption $\left|x_\alpha\right| \goeseta 0$. It follows from (ii) that $\left|x_\alpha\right| \goesseta 0$, and since $\mathrm{s}\eta$ is locally solid, we conclude that $x_\alpha \goesseta 0$.\medskip

(iv):  Necessity in the first statement follows immediate from (i) and Remark \ref{Rem: Locally full}, and sufficiency follows from (ii).  For the second statement, let $(x_\alpha)$ be a net such that $x_\alpha =0$ for every $\alpha$.  Suppose that $x_\alpha \goesseta x$.  Then $|x-x_\alpha|\goeseta 0$.  Since $\eta$ is Hausdorff and $|x-x_\alpha|=|x|$ for every $\alpha$, it follows that $0=|x|=x$.  Hence $\{0\}$ is $\mathrm{s}\eta$-closed so that $\mathrm{s}\eta$ is Hausdorff by \cite[Proposition 3.1.2]{Beattie:02}.
\end{proof}

\begin{prop}\label{Prop:  Locally solid mod countability}
Let $\eta$ be a linear convergence structure on $X$.  If $\eta$ is (weakly or strongly) first countable, then so is $\mathrm{s}\eta$.
\end{prop}

\begin{proof}
Assume that $\eta$ is weakly first countable, let $x_\alpha \goesseta 0$, and let $(y_{\beta})$ be given by part (i) of Proposition \ref{Prop:  Locally solid mod properties}. Since $\eta$ is weakly first countable, there is $\left(\beta'_{n}\right)$ such that $y_{\beta_{n}}\goeseta 0$ whenever $\beta_{n}\ge\beta'_{n}$, for every $n\in\N$. Let $(\alpha'_{n})$ be an increasing sequence in $A$ such that for every $\alpha \ge \alpha'_n$ there exists $\beta\ge \beta'_n$ so that $|x_\alpha|\le y_\beta$. Assume that $(\alpha_{n})$ is such that $\alpha_{n}\ge\alpha'_{n}$ for every $n\in\N$. For each $n\in \N$, let $\beta_{n}\ge\beta'_{n}$ be such that $|x_{\alpha_{n}}|\le y_{\beta_{n}}$. By assumption we have $0\le y_{\beta_{n}}\goeseta 0$, hence, by part (ii) of Proposition \ref{Prop:  Locally solid mod properties} we have $y_{\beta_{n}}\goesseta 0$, and since $\mathrm{s}\eta$ is locally solid, we conclude that $x_{\alpha_{n}}\goeseta 0$.

Suppose that $\eta$ is first countable, and let $\mathcal{F}\goesseta 0$.  By definition, there exists $X_+\in\mathcal{G}\goeseta 0$ so that $\Sol(\mathcal{G})\subseteq \mathcal{F}$.  By the first countability of $\eta$, there exists a countable filter base $\mathcal{B}=\{B_{n}\}_{n\in\N}\subseteq \mathcal{G}$ so that $[\mathcal{B}]\goeseta 0$. Replacing each $B_n$ with $B_n\cap X_+$ if necessary, we may suppose that $X_+\in [\mathcal{B}]$.  Then $\Sol([\mathcal{B}])\goesseta 0$ and so, since $\Sol([\mathcal{B}]) \subseteq \Sol(\mathcal{G})\subseteq \mathcal{F}$, it follows that $\mathrm{s}\eta$ is first countable.

Assume that $\eta$ is strongly first countable, and let $\mathcal{H}$ be a countable fixed intersection-closed local basis for $\eta$ at $0$.  Let $\mathcal{H}'\defeq \{\Sol(H\cap X_+) \mid H\in\mathcal{H}\}$.  We show that $\mathcal{H}'$ is a local basis for $\mathrm{s}\eta$ at $0$.  Let $\mathcal{F}\goesseta 0$. There exists a filter $X_+\in\mathcal{G}\goeseta 0$ so that $\Sol(\mathcal{G})\subseteq \mathcal{F}$.  It follows from $\mathcal{G}\goeseta 0$ that $[\mathcal{H}\cap \mathcal{G}]\goeseta 0$.  Since $X_+\in \mathcal{G}$, it follows that $\mathcal{D}\defeq \{H\cap X_+ \mid H\in\mathcal{G}\cap \mathcal{H}\}$ is a filter base. Note that if $H\in\mathcal{G}$, then $H\cap X_+\in \mathcal{G}$, hence $\Sol \left(H\cap X_+\right)\in \mathcal{F}$. Thus, $\Sol(\mathcal{D})\subseteq \mathcal{F}$. At the same time, $\Sol(\mathcal{D})\subseteq \mathcal{H}'$, hence $\Sol\left([\mathcal{D}]\right)\subseteq\left[\Sol(\mathcal{D})\right]\subseteq[\mathcal{H}'\cap \mathcal{F}]$. Finally, $[\mathcal{H}\cap\mathcal{G}]\subseteq [\mathcal{D}]$ implies $X_+\in [\mathcal{D}]\goeseta 0$. We conclude that $[\mathcal{H}'\cap \mathcal{F}]\goesseta 0$.
\end{proof}

\section{First countability of order and related convergences}\label{Sec: order}

Throughout this section we denote by $X$ an Archimedean vector lattice.  One of the most important examples of a locally solid convergence is \term{order convergence}: a net $(x_\alpha)$ in $X$ \term{converges in order} to $x$ (written $x_\alpha\goeso x$) if there exists a set $D$ such that $\bigvee D= 0$ and
for every $d\in D$ there exists $\alpha_d\in A$ such that $\abs{x_\alpha-x}\le d$ for all $\alpha\ge\alpha_d$.  By replacing $D$ with $D^\wedge$ if necessary, one may suppose that $D\downarrow 0$.  In the language of filters, $\mathcal F\goeso x$ if $\mathcal F$ contains a collection of order intervals whose intersection is~$\{x\}$.\medskip

Closely related to order convergence are \term{$\sigma$-order convergence} and \term{relatively uniform convergence}.  A net $(x_\alpha)$ in $X$ \term{$\sigma$-order converges} to $x$, written $x_\alpha\goesso x$, if there exists a sequence $(u_n)$ in $X_+$ such that $\bigwedge\limits_{n\in\N} u_n= 0$ and for every $n\in N$ there exists $\alpha_n$ such that $\abs{x_\alpha-x}\le u_n$ whenever $\alpha\ge\alpha_n$.  We will call $(u_n)$ a \term{dominating sequence} for the net $(x_\alpha)$.  Replacing $(u_n)$ with $(u_{1}\wedge ...\wedge u_{n})$ we may assume that $(u_n)$ is decreasing.  In filter language, $\mathcal F\goesso x$ if $\mathcal F$ contains a sequence of order intervals whose intersection is $\{x\}$.  It is easy to see that $\sigma{\rm o}$ is a first countable Hausdorff and locally order bounded locally solid convergence structure and it is stronger than order convergence. The convergence structure $\sigma\mathrm{o}$ was investigated in \cite{Anguelov:05}.  A net $(x_\alpha)$ in $x$ converges \term{relatively uniformly} to $x$, written $x_\alpha \goesu x$, if there exists $u\in X_+$ so that for every $\varepsilon>0$ there exists $\alpha_\varepsilon\in A$ so that $|x-x_\alpha|\le \varepsilon u$ for all $\alpha \ge \alpha_\varepsilon$.  Relative uniform convergence is a locally solid convergence structure, see \cite[Example 2.3 \& Proposition 2.4]{VanderWalt:13} and \cite[Section 5]{OBrien:23}, which is stronger than $\sigma$-order convergence.  All three introduced convergences are locally order bounded.

This section contains the main results of the paper, namely, characterisations of $X$ for which ${\rm o}$, $\sigma{\rm o}$, ${\rm ru}$ and their unbounded modifications are (strongly) first countable.  We first consider ${\rm ru}$ and ${\rm uru}$, which are the easiest cases to deal with.  The following result follows immediate from Corollary \ref{first-c-lemma} and Proposition \ref{fsc-unbdd}, see also \cite{VanderWalt:16}.

\begin{prop}\label{ru-fc}
Relative uniform convergence is first countable. Furthermore, the following statements are equivalent. \begin{enumerate}[(i)]

    \item There exists a countable set $C\subseteq X_+$ so that $X=I_C$;

    \item ${\rm ru}$ is strongly first countable.

\end{enumerate}

If either (i) or (ii), and hence both, is true, then ${\rm uru}$ is strongly first countable.

\end{prop}

Note that the converse to the last claim is false. Indeed, on $X=c_0$ ${\rm uru}$ coincides with the pointwise topology which is metrizable, but $X$ does not satisfy the condition (i) of the Proposition, since it does not have a strong unit (see Remark \ref{rem-strong}).\medskip

We now turn to order- and $\sigma$-order convergence.  Our first result characterises when order convergence on $X$ is first countable.  Recall that $\sigma$-order convergence is always first countable.  We say that $X$ has the \term{countable supremum property (CSP)} if for every set $A\subseteq X$ such that $\bigwedge A=0$, there is a countable set $B\subseteq A$ such that $\bigwedge B=0$.  This property is equivalent to the fact that every order bounded, disjoint subset of $X$ is countable, see \cite[Theorem 29.3]{Luxemburg:71} or \cite[Theorem 2.23]{Deng:25}. In the literature, this property is sometimes referred to as `order separability', see for instance \cite[Definition 23.2 (iii)]{Luxemburg:71}.  However, in the context of topological spaces and, more generally, convergence spaces, a separable space is one that contains a countable dense set; hence `order separable' may be confused with `separable with respect to the order convergence structure' but, in general, these notions are distinct.  Indeed, if $T$ is an uncountable set then $c_{00}(T)$ has the CSP, but if $A\subseteq c_{00}(T)$ is countable, there exists $t\in T$ so that $u(t)=0$ for every $u\in A$; then $\left\{v\in c_{00}(T),~ v(t)=0\right\}$ is a pointwise closed, hence order closed proper subset of $c_{00}(T)$, which contains $A$.  See Remarks \ref{sep1}, \ref{sep2} and \ref{sep3} for more details on the connection between CSP and various separability conditions.


\begin{prop}\label{csp}
The following are equivalent:
\item[(i)] $X$ has CSP;
\item[(ii)] $\mathrm{o}=\sigma \mathrm{o}$ on $X$;
\item[(iii)] Order convergence is first countable on $X$;
\item[(iv)] Order convergence is weakly first countable on $X$;
\item[(v)] Order convergence is Fr\'{e}chet-Urysohn on $X$;
\item[(vi)] Order convergence is sequential on $X$.
\end{prop}

\begin{proof}
(i)$\Rightarrow$(ii) is straightforward, and (ii)$\Rightarrow$(iii) follows from the fact that $\sigma \mathrm{o}$ is always first countable.  Since every first countable space is weakly first countable, every weakly first countable space is is Fr\'{e}chet-Urysohn and every Fr\'{e}chet-Urysohn space is sequential, (iii)$\Rightarrow$(iv)$\Rightarrow$(v)$\Rightarrow$(vi) also hold. For (vi)$\Rightarrow$(i) see \cite[Theorem 3.8 and Remark 3.9]{Deng:25}.
\end{proof}

\begin{rem}\label{bcsp}
In the conditions of Proposition \ref{csp} the clause ``(on) $X$'' can be replaced with ``(on) every order interval of $X$''.  This apparently weakens each of the conditions, but it is easy to see that $X$ has the CSP if and only if each of its order intervals does.  Furthermore, the proof of (vi)$\Rightarrow$(i) only utilizes the fact that if $Y$ is a $\sigma$-ideal, then $\left[0,x\right]\cap Y$ is sequentially order closed for every $x\in X$.
\qed\end{rem}

In the remainder of this section we investigate strong first countability of order- and $\sigma$-order convergence on vector lattices.  We call $X$ \term{strongly first countable} if order convergence is strongly first countable on $X$, and \term{$\sigma$-strongly first countable} if $\sigma$-order convergence is strongly first countable.  The interrelationship between these concepts is described as follows.

\begin{cor}\label{main0}
$X$ is strongly first countable if and only if it is $\sigma$-strongly first countable and has the CSP.
\end{cor}
\begin{proof}
It follows immediately from Proposition \ref{csp} that if $X$ has the CSP, then strong first countability and $\sigma$-strong first countability are equivalent for $X$.  Since strong first countability  implies first countability, it follows from Proposition \ref{csp} that if $X$ is strongly first countable, then it has the CSP.
\end{proof}

\begin{rem}\label{Rem:  Strong order conv}
The following alternative definition of order convergence used to be dominant in the literature (see e.g. \cite{Aliprantis:03} and \cite{Aliprantis:06}). We say that a net $(x_{\alpha})\subseteq X$ \term{strongly order converges} to $x\in X$ if there are $\alpha_{0}$ and a net $(y_\alpha)_{\alpha\ge\alpha_{0}}$ such that $y_\alpha \downarrow 0$ and $|x_\alpha-x|\le y_\alpha$, for every $\alpha\ge\alpha_{0}$. It is easy to see that this implies $x_\alpha\goeso x$. Moreover, for sequences strong order convergence agrees with $\sigma$-order convergence.

One can prove that strong order convergence satisfies the requirements for a convergence considered in \cite{Aydin:21}. Moreover, it satisfies requirements (i) and (iii) of a convergence structure. Let us show that it does not satisfy condition (ii), by providing an example of a net which strongly order converges to $0$, but has a quasi-subnet that is a sequence which does not $\sigma$-order converge to $0$.

Let $X:=c\left(\R\right)$ and let $A$ be the collection of finite subsets of $\R$ which contain $1$. For $\alpha\in A$ define $y_{\alpha}:=\1_{\R\backslash\alpha}$ and $x_{\alpha}:=\delta_{\bigvee\left(\alpha\cap\N\right)+1}$. It is easy to see that $x_{\alpha}\le y_{\alpha}\downarrow \0$. It is well-known that $(\delta_{n})$ is not $\sigma$-order null. However, this sequence is a quasi-subnet of $(x_\alpha)$. Indeed, for any $\alpha_0$ let $n_0:=\bigvee\left(\alpha\cap\N\right)+1$; if $n\ge n_0$, then for $\alpha:=\alpha_0\cup\left\{n-1\right\}$ we have $x_\alpha=\delta_n$, and so $\left\{\delta_{n}\right\}_{n\ge n_{0}}\subseteq\left\{x_{\alpha}\right\}_{\alpha\ge \alpha_{0}}$.
\qed\end{rem}

\subsection{Passing to and from sublattices}

Many locally solid convergence structures $\eta$, such as order convergence, can be defined on an arbitrary Archimedean vector lattice $X$. If we wish to emphasise the underlying space, we write $\eta_{X}$. If $Y$ is a sublattice of $X$, we write $\eta_{X}|Y$ to denote the convergence structure $\eta$ on $X$ restricted to $Y$.  In general, $\eta_{X}|Y$ is distinct from $\eta_{Y}$.\medskip

We show next that ($\sigma$-)strong first countability is a ``local'' property:  it passes down to principal ideals, and can be lifted to the whole space from a suitable sequence of principal ideals.  The proofs will be given for the ``$\sigma$'' case, as the other case is proven similarly.

\begin{lem}\label{so-Iu-int}
Let $u\in X_+$.  Then $\sigma \mathrm{o}_X|_{[-u,u]}=\sigma \mathrm{o}_{I_u}|_{[-u,u]}$ and $\mathrm{o}_X|_{[-u,u]} = \mathrm{o}_{I_u}|_{[-u,u]}$.
\end{lem}

\begin{proof}
  It is clear that $\sigma \mathrm{o}_X\le\sigma \mathrm{o}_{I_u}$ on $I_u$ and therefore on $A$.  To prove the converse, suppose that $x_\alpha\xrightarrow{\sigma \mathrm{o}_X}x$ for some $(x_\alpha)$ and $x$ in $[-u,u]$. Then there exists a dominating sequence $(y_n)$ in $X$ for this convergence. It follows from $\abs{x_\alpha-x}\le 2u$ for all $\alpha$ that the sequence  $y_n\wedge(2u)$ is still dominating. Since it is contained in $I_u$, we conclude that $x_\alpha\xrightarrow{\sigma \mathrm{o}_{I_u}}x$.
\end{proof}

It now follows that ($\sigma$-)strong first countability passes down to principal ideals.

\begin{cor}\label{so-Iu-sfc}
Every principal ideal of a ($\sigma$-)strongly first countable vector lattice is ($\sigma$-)strongly first countable.
\end{cor}

\begin{proof}
Fix $u\ge 0$ and let $n\in\mathbb N$.  It follows from Lemma \ref{so-Iu-int} that $\sigma \mathrm{o}_X$ and $\sigma \mathrm{o}_{I_u}$ agree on $[-nu,nu]$. Since $\sigma \mathrm{o}_X$ is strongly first countable on $X$, it is also strongly first countable on $[-nu,nu]$.  Hence $\sigma \mathrm{o}_{I_u}$ is strongly first countable on $[-nu,nu]$.\medskip

Let $j_n\colon[-nu,nu]\to I_u$ be the inclusion map, where each $[-nu,nu]$ is equipped with the restriction of $\sigma \mathrm{o}_{I_u}$ or, equivalently, of $\sigma \mathrm{o}_X$.  We claim that $\sigma \mathrm{o}_{I_u}$ is the final convergence structure with respect to $\{j_n\}_{n\in\N}$, hence also strongly first countable.  Indeed, for every $n\in\N$, the map $j_n$ is continuous by definition.  Let $\eta$ be a convergence structure on $I_u$ which makes all the $j_n$'s continuous.  It suffices to show that $\sigma \mathrm{o}_{Iu}\ge\eta$.  Suppose that $x_\alpha\xrightarrow{\sigma \mathrm{o}_{I_u}}x$ for some $(x_\alpha)$ and $x$ in $I_u$.  Passing to a tail, we may assume that $(x_\alpha)$ is order bounded in $I_u$, hence there exists $n\in \N$ such that $(x_\alpha)$ and $x$ are in $[-nu,nu]$.  Since $j_n$ is continuous, we conclude that $x_\alpha\goeseta x$.
\end{proof}

\begin{prop}\label{sos1c}
$X$ is ($\sigma$-)strongly first countable if and only if there exists an increasing sequence $(v_n)$ in $X_+$ such that $X=I_{\{v_n\}}$ and $I_{v_n}$ is ($\sigma$-)strongly first countable for every $n\in \N$.
\end{prop}

\begin{proof}
Suppose that $X$ is $\sigma$-strongly first countable. By Proposition \ref{first-c-lemma}, there exists an increasing sequence $(v_n)$ in $X_+$ such that $X=I_{\{v_n\}}$. By Corollary \ref{so-Iu-sfc}, $\sigma \mathrm{o}_{I_{v_n}}$ is strongly first countable for every $n\in \N$.\medskip

To prove the converse, equip each $I_{v_n}$ with $\sigma \mathrm{o}_{I_{v_n}}$, which is strongly first countable by assumption, and let $j_n\colon I_{v_n}\to X$ be the inclusion map.  We claim that $\sigma \mathrm{o}_X$ is the final convergence structure with respect to $\{j_n\}_{n\in\N}$; this yields that $\sigma \mathrm{o}_X$ is strongly first countable.

It is clear that if $x_\alpha\xrightarrow{\sigma \mathrm{o}_{I_{v_n}}}x$ in $I_{v_n}$ then $x_\alpha\xrightarrow{\sigma \mathrm{o}_{X}}x$ in $X$, hence every map $j_n$ is $\sigma \mathrm{o}_{I_{v_n}}$-to-$\sigma \mathrm{o}_X$ continuous.  Now let $\eta$ be a convergence structure on $X$ which makes all the $j_n$'s continuous. Suppose that $x_\alpha\xrightarrow{\sigma \mathrm{o}_X}x$ for some $(x_\alpha)$ and $x$ in $X$.  it suffices to show that $x_\alpha\goeseta x$. Let $(y_m)$ be a dominating sequence for $(x_\alpha)$.  Find $n\in\N$ such that $x,y_1\in I_{v_n}$.  It follows that a tail of $(x_\alpha)$ is contained in $I_{v_n}$ and converges there to $x$ with respect to $\sigma o_{I_{v_n}}$.  By the continuity of $j_n$, we conclude that $x_\alpha\goeseta x$.
\end{proof}

\begin{cor}
If $X$ is ($\sigma$-)strongly first countable and admits a locally solid completely metrizable topology, then it has a strong unit.
\end{cor}

We now investigate how ($\sigma$-)strong first countability is transferred to variously dense sublattices.  Recall that a sublattice $Y$ of $X$ is \term{order dense} in $X$ if for every $0<x\in X$ there exists $A\subseteq Y_+$ so that $x=\bigvee A$.  We call $Y$ \term{super order dense} in $X$ if for every $0<x\in X$ there exists $\left\{y_{n}\right\}_{n\in\N}\subseteq Y_{+}$ such that $x=\bigvee\limits_{n\in\N} y_{n}$.

\begin{prop}\label{dense}
Let $Y\subseteq X$ be a majorizing sublattice. \begin{enumerate}[(i)]

    \item[(i)] If $Y$ is order dense in $X$, then $\mathrm{o}_{Y}=\mathrm{o}_{X}|_{Y}$. Moreover, $X$ is strongly first countable if and only if $Y$ is strongly first countable.

    \item[(ii)] If $Y$ is super order dense in $X$, then $\sigma\mathrm{o}_{Y}=\sigma\mathrm{o}_{X}|_{Y}$ and, moreover, $X$ is $\sigma$-strongly first countable if and only if $Y$ is $\sigma$-strongly first countable.

\end{enumerate}
\end{prop}
\begin{proof}
Since the proofs of (i) and (ii) follow similar arguments, we only prove (ii).  Suppose that $Y$ is super order dense in $X$.  Observe that since $Y$ is majorizing in $X$, for every $x\in X$ there exists $\left\{y_{n}\right\}_{n\in\N}\subseteq Y$ such that $x=\bigwedge\limits_{n\in\N} y_{n}$.  Indeed, let $x\in X$ and $x<y\in Y$.  Then $0<y-x$ so there exists $\{z_n\}_{n\in\N}\subseteq Y$ so that $y-x=\bigvee\limits_{n\in\N} z_{n}$.  Hence $x=\bigwedge\limits_{n\in\N} y-z_{n}$.

Let $(y_\alpha)$ be a net in $Y$.  Suppose that $y_\alpha\xrightarrow{\sigma \mathrm{o}_Y}0$.  Let $(u_n)$ be a dominating sequence in $Y$.  Since $Y$ is order dense in $X$, we have $u_n\downarrow 0$ in $X$, hence $(u_n)$ is a dominating sequence in $X$ so that $y_\alpha\xrightarrow{\sigma \mathrm{o}_{X}}0$.

Conversely, suppose that $y_\alpha\xrightarrow{\sigma \mathrm{o}_{X}}0$.  Let $(v_n)$ be a dominating sequence in $X$.  For every $n\in\mathbb N$ there exists a countable set $A_n\subseteq Y$ so that $\bigwedge A_n=v_n$.  Then $\bigwedge \bigcup\limits_{n=1}^{\8} A_n=0$, and this set witnesses $y_\alpha\xrightarrow{\sigma \mathrm{o}_{Y}}0$.

It now follows immediately that if $X$ is $\sigma$-strongly first countable then $Y$ is $\sigma$-strongly first countable.  Let us prove the converse.\medskip

Assume that $Y$ is strongly first countable.  Let $\mathcal H$ be a countable intersection-closed local base for $\sigma\mathrm{o}_Y$ at $0$.  Let
  \begin{math}
    \mathcal H'=\bigl\{\Sol(A)\mid A\in\mathcal H\bigr\},
  \end{math}
where $\Sol(A)$ is the solid hull of $A$ in $X$.  We claim that $\mathcal H'$ is a countable local basis for $\sigma\mathrm{o}_{X}$.

Suppose that $\mathcal F\xrightarrow{\sigma\mathrm{o}_{X}}0$.  Then there exists a sequence $u_n\downarrow 0$ in $X$ such that $[-u_n,u_n]\in\mathcal F$ for every $n\in \N$.  For each $n$ find a  countable set $B_{n}\subseteq Y$ such that $\bigwedge B_{n}= u_n$, and let $B\defeq\bigcup\limits_{n=1}^{\8} B_n$.  It follows that $\bigwedge B=0$ in $X$ and, therefore, in $Y$.  Let $\mathcal G$ be the filter on $Y$ generated by $\{\bigl[-b,b\bigr]_Y \mid b\in B^\wedge\}$, the subscript $Y$ indicates that the intervals are taken in $Y$.  Since $B^\wedge$ is countable and $\bigwedge B^\wedge = \bigwedge B = 0$, it follows that $\mathcal G\xrightarrow{\sigma\mathrm{o}_{Y}}0$.  By assumption, $[\mathcal G\cap\mathcal H]\xrightarrow{\sigma\mathrm{o}_{Y}}0$.  Hence there exists a sequence $(z_k)$ in $Y$ such that $z_k\downarrow 0$ in $Y$ and, therefore, in $X$, and $[-z_k,z_k]_Y\in[\mathcal G\cap\mathcal H]$ for all $k$.  Therefore, for every $k\in\N$, there exists $A_k\in\mathcal G\cap\mathcal H$ such that $A_k\subseteq[-z_k,z_k]_Y$.

Let $k\in \N$.  It follows from $A_k\in\mathcal G$ that $A_k$ contains $\bigl[-b,b\bigr]_Y$ for some $b\in B^\wedge$.  But there exists $n\in\N$ so that $u_n\le b$, so
  \begin{displaymath}
    [-u_n,u_n]\subseteq\bigl[-b,b\bigr]
    \subseteq\Sol(A_k)\subseteq[-z_k,z_k],
  \end{displaymath}
where the intervals are in $X$.  It follows that $\Sol(A_k)$ is in $\mathcal F$ and, therefore, in $\mathcal F\cap\mathcal H'$.  We conclude that $[-z_k,z_k]\in[\mathcal F\cap\mathcal H']$.  Thus, $[\mathcal F\cap\mathcal H']\xrightarrow{\sigma\mathrm{o}_{X}}0$.
\end{proof}

The following two results will be used in Section \ref{Section:  The main characterisation}.  The latter is an immediate consequence of Proposition \ref{dense}.

\begin{lem}\label{cor}
If $X$ is a unital norm dense sublattice of $C(K)$, then it is super order dense.
\end{lem}
\begin{proof}
Let $f\in C(K)$ and $n\in\mathbb N$. Find $v_n\in X$ such that $\|v_n-f\|<\frac{1}{2n}$. Let $y_n=(v_n-\frac{1}{2n}\1)^+$, then $y_n\in X_+$, and $f-\frac1n\1\le y_n\le f$ yields $\bigvee\limits_{n\in\N}y_{n}= f$.
\end{proof}

\begin{cor}\label{corr}
If $X$ is a unital norm dense sublattice of $C(K)$, then it is ($\sigma$-)strongly first countable if and only if $C(K)$ is ($\sigma$-)strongly first countable.
\end{cor}

\subsection{A characterisation of $\sigma$-strong first countability}

We introduce a condition on $X$ that characterises $\sigma$-strong first countability.  Here we denote by $X^\delta$ the Dedekind completion of $X$.

\begin{thm}\label{sigmastrong}
$X$ is $\sigma$-strongly first countable if and only if there exists a countable set $C$ in $X_+^\delta$ such that for every $x_n\downarrow 0$ in $X$ there exists $D\subseteq C$ so that $\bigwedge_{X^{\delta}} D=0$ and for every $d\in D$ there exists $n\in\mathbb N$ such that $x_n\le d$.
\end{thm}
\begin{proof}
Necessity: Let $\mathcal H=\{H_{n}\}$ be a countable local basis at $0$ for the $\sigma$-order convergence on $X$. According to the comment before Proposition \ref{muB-fc0} we may assume that $H_{n}$ is order bounded, for every $n\in\N$. Put
  \begin{displaymath}
    C:=\bigl\{\bigvee\nolimits_{X^{\delta}}\abs{H}\mid H\in\mathcal H\bigr\},
  \end{displaymath}
  where $\abs{H}=\bigl\{\abs{h}\mid h\in H\bigr\}$. Let $x_n\downarrow 0$ in $X$. Let $\mathcal F$ be the filter generated by
  \begin{math}
    \bigl\{[-x_n,x_n]_{X}\mid n\in\mathbb N\bigr\}.
  \end{math}
  Clearly, $\mathcal F\goesso 0$ and, therefore, $[\mathcal F\cap\mathcal H]\goesso 0$. Then there is a sequence $(y_m)$ in $X$ such that $y_m\downarrow 0$ and $[-y_m,y_m]\in[\mathcal F\cap\mathcal H]$ for every $m$. For every $m$, fix some $n_{m}$ such that $H_{n_{m}}\in\Fo$ and $H_{n_{m}}\subseteq[-y_m,y_m]$; put $d_m:=\bigvee_{X^{\delta}}\abs{H_{n_{m}}}$. Let $D:=\{d_m\mid m\in\mathbb N\}\subseteq C$. For every $m$, $H_{n_{m}}\in\mathcal F$ yields $[-x_n,x_n]\subseteq H_{n_{m}}$ for some $n$. It now follows from
  \begin{math}
    [-x_n,x_n]\subseteq H_{n_{m}}\subseteq[-y_m,y_m]
  \end{math}
  that $x_n\le d_m\le y_m$. As $y_m\downarrow 0$ conclude that $\bigwedge_{X^{\delta}} D=0$.\medskip

Sufficiency: Define a convergence structure $\eta$ on $X^\delta$ as follows:  $\mathcal{F}\goeseta 0$ if there exists a sequence $x_n\downarrow 0$ in $X$ so that $[-x_n,x_n]_{X^\delta}\in\mathcal{F}$ for every $n\in \N$; equivalently, $x_\alpha \goeseta 0$ if there exists a sequence $x_n\downarrow 0$ in $X$ so that $[-x_n,x_n]_{X^\delta}$ contains a tail of $(x_\alpha)$ for every $n\in \N$.  For $x\in X^\delta\setminus \{0\}$ we let $\mathcal{F}\goeseta x$ if $|\mathcal{F}-x|\goeseta 0$; equivalently, $x_\alpha \goeseta x$ if $|x_\alpha-x|\goeseta 0$.  Using \cite[Theorem 2.1]{Bilokopytov:23} it is easy to see that $\eta$ is a locally solid convergence structure on $X^\delta$.

Clearly, $\eta|X=\sigma{\rm o}$, so it suffices to show that that $\eta$ is strongly first countable.  Let $C$ be as in the statement of the theorem, and let $C^\wedge$ be the set of all finite infima of elements of $C$. Put
\begin{math}
    \mathcal H=\bigl\{[-c,c]_{X^\delta}\mid c\in C^\wedge\bigr\}.
  \end{math}
  Then $\mathcal H$ is countable and closed under finite intersections. Suppose $\mathcal F\goeseta 0$. There exists a sequence $(x_n)$ such that $x_n\downarrow 0$ in $X$ and $[-x_n,x_n]_{X^\delta} \in\mathcal F$ for every $n\in \N$.  Let $D$ be as in the statement of the theorem. For every $d\in D$, there exists $n\in\N$ such that $x_n\le d$, and, therefore, $[-d,d]\in\mathcal F$. Clearly, $[-d,d]\in\mathcal H$.  It follows from $\bigwedge D=0$ that $[\mathcal F\cap\mathcal H]\goeseta 0$.
\end{proof}

We proceed with a slight modification of the condition from the preceding result which is a convenient sufficient condition for $\sigma$-strong first countability.  This condition, when applied to a suitable $C(K)$ space, enables us to show that $\sigma$-strong first countability does not imply strong first countability.
\[
\parbox[c]{9.5cm}{There exists a countable set $C$ in $X_+$ such that for every $x_n\downarrow 0$ there exists $D\subseteq C$ so that $\bigwedge D=0$ and for every $d\in D$ there exists $n\in\mathbb N$ such that $x_n\le d$.} \label{so-sfc-CD} \tag{$\star$}
\]

In other words $C$ is such that for every sequence $x_n\downarrow 0$ in $X$ we have \linebreak $\bigwedge\bigcup\limits_{n\in\N}\left(C\cap\left[x_{n},+\8\right)\right)=0$. It follows immediately from Theorem \ref{sigmastrong} that \eqref{so-sfc-CD} implies $\sigma$-strong first countability of $X$, and the two conditions are equivalent if $X$ is order complete. As we show below, \eqref{so-sfc-CD} is also equivalent with $\sigma$-strong first countability if $X$ is \term{almost $\sigma$-complete}; that is, there exists a $\sigma$-order complete space $Y$ containing $X$ is a super order dense sublattice.

In order to verify the claims made above, we require the following results.

\begin{prop}\label{odmstar}
Let $Y$ be a super order dense and majorizing sublattice of $X$.  Then $X$ satisfies \eqref{so-sfc-CD} if and only if $Y$ does.
\end{prop}

\begin{proof}
Assume that $Y$ satisfies \eqref{so-sfc-CD}, and let $C\subseteq Y_+$ be as given by the condition.  Let $x_n\downarrow 0$ in $X$.  As shown in the proof of Proposition \ref{dense}, for every $n\in \N$ there exists a sequence $y_k^{(n)} \downarrow x_n$ in $Y$. For $m\in\N$ let $z_{m}:=\bigwedge\limits_{n,k\le m} y_k^{(n)}\ge x_m$. Then, $z_m\downarrow 0$. Let $D$ be as in \eqref{so-sfc-CD}; we have $\bigwedge D=0$ in $Y$, hence in $X$.  For every $d\in D$ there is $z_m$ such that $x_m\le z_m\le d$.  Therefore $X$ satisfies \eqref{so-sfc-CD}.

Now suppose that $X$ satisfies \eqref{so-sfc-CD}, and let $C\subseteq X_+$ be as in the condition.  For every $c\in C$ there exists a countable set $A_c\subseteq Y$ so that $c= \bigwedge A_c$.  Define $A \defeq \bigcup\limits_{c\in C}^{} A_c$.  Clearly, $A\subseteq Y_+$ is countable.  Let $y_n\downarrow 0$ in $Y$.  Since $Y$ is super order dense in $X$ we also have $y_n\downarrow 0$ in $X$.  Therefore there exists $D\subseteq C$ so that $\bigwedge D=0$ and for every $n\in \N$ there exists $d_n\in D$ so that $y_n\le d_n$.  Let $B \defeq \bigcup\limits_{d\in D}^{} A_d$. Then $B\subseteq A$, $\bigwedge B=0$, and $y_n \le d_n\le b$ for every $n\in\N$ and $b\in A_{d_n}\subseteq B$.  Therefore $Y$ satisfies \eqref{so-sfc-CD}.
\end{proof}

\begin{lem}\label{countcard}
Let $X$ be a set endowed with a strongly first countable Hausdorff convergence structure.  Assume that $x_{0}\in X$ is such that any countably-indexed non-terminating net in $Y:=X\backslash\left\{x_{0}\right\}$ consisting of mutually distinct elements converges to $x_{0}$. Then $X$ is countable.
\end{lem}
\begin{proof}
For the sake of obtaining a contradiction, suppose that $X$ is not countable.  Let $\mathcal{H}=\left\{H_{n}\right\}_{n\in\N}$ be a fixed intersection-closed local basis at $x_{0}$.  Let $Y_{n}:=Y\backslash H_{n}$, $n\in\N$.  Since $X$ is Hausdorff, $\mathcal{H}$ is fixed, and $[\mathcal{H}]\to x_0$ by Proposition \ref{prop:  Existence of convergent directed fixed local basis}, it follows that $\displaystyle\bigcap \mathcal{H}=\{x_0\}$.  We may therefore suppose that $Y\supseteq Y_n\neq \varnothing$ for every $n\in\N$.  In addition, it follows that $Y=\bigcup\limits_{n\in \N}Y_{n}$.  There exists a countable set $Z_1\subseteq Y$ so that $Z_1\cap Y_n\neq \varnothing$ for every $n\in\N$.  We inductively construct a disjoint sequence $\left(Z_{m}\right)_{m\in\N}$ of countable, nonempty subsets of $Y$ with the following property:  For all $m,n\in\N$, $Z_m$ intersects $Y_{n}\backslash\bigcup\limits_{k=1}^{m-1}Z_{k}$, provided that the latter is not empty.  Assume that $Z_{1},\ldots,Z_{m-1}$ have been constructed.  Since $Z_1,\dots,Z_{m-1}$ are countable and $Y=\bigcup\limits_{n\in \N}Y_{n}$ is uncountable, the existence of $Z_m$ follows from the axiom of countable choice.

Let $\Gamma:=\left\{\left(n,z\right),~ n\in\N,~ z\in Z_{n}\right\}$ be ordered by the first component, and for every $\gamma=(n,z)\in\Gamma$, define $x_{\gamma}:=z$.  Then $\left(x_{\gamma}\right)$ is a countably-indexed non-terminating net in $Y$ consisting of mutually distinct elements.  By assumption, $x_{\gamma}\to x_{0}$, and so the tail filter $\Fo$ of $\left(x_{\gamma}\right)$ converges to $x_{0}$.  It is easy to see that $\Fo=\left[\left\{\bigcup\limits_{k\ge n}Z_{k} \mid n\in\N\right\}\right]$.  Let $N:=\left\{n\in\N,~ H_{n}\in \Fo\right\}$. By our assumption $\left[\{H_{n} \mid n\in N\}\right]=\left[\Fo\cap \mathcal{H}\right]\to x_{0}$.  Since $X$ is Hausdorff and $\mathcal{H}$ is fixed, $\bigcap\limits_{n\in N}H_{n}=\left\{x_{0}\right\}$.

For every $n\in N$ there exists $m\in\N$ such that $\bigcup\limits_{k\ge m}Z_{k}\subseteq H_{n}$, hence $Z_{m}\subseteq H_{n}$, so that $Z_{m}\cap Y_{n}=\varnothing$.  It therefore follows from the construction of $\left(Z_{m}\right)_{m\in\N}$ that $Y_{n}\backslash\bigcup\limits_{k=1}^{m-1}Z_{k}=\varnothing$ so that $Y_{n}$ is countable.  It now follows that $Y=Y\backslash \bigcap\limits_{n\in N}H_{n}=\bigcup\limits_{n\in N}\left(Y\backslash H_{n}\right)=\bigcup\limits_{n\in N}Y_{n}$ is countable, a contradiction.
\end{proof}

\begin{lem}\label{oso}
If $X$ is $\sigma$-order complete, then $\mathrm{o}$ and $\sigma\mathrm{o}$ agree on countably indexed nets.
\end{lem}
\begin{proof}
Since $\mathrm{o}\le\sigma\mathrm{o}$ we only need to prove that if $\Gamma$ is countable, and $x_{\gamma}\goeso 0$, then $x_{\gamma}\goesso 0$. WLOG we may assume that $\left(x_{\gamma}\right)$ is order bounded, and so $v_{\gamma}:=\bigvee\limits_{\beta\ge\gamma}x_{\gamma}$ and $w_{\gamma}:=\bigwedge\limits_{\beta\ge\gamma}x_{\gamma}$ exist and satisfy $\bigwedge v_{\gamma}=0=\bigvee w_{\gamma}$. For every $\gamma$ we have $v_{\gamma}\ge x_{\gamma}\ge w_{\gamma}$, hence $\left|x_{\gamma}\right|\le w_{\gamma}-v_{\gamma}$. Then $w_{\gamma}-v_{\gamma}\downarrow 0$ witnesses $x_{\gamma}\goesso 0$.
\end{proof}

\begin{prop}\label{so-sfc}
If $X$ is $\sigma$-order complete and $\sigma$-strongly first countable, then it has the CSP and is order complete.
\end{prop}
\begin{proof}
It is well-known that any order bounded disjoint non-terminating net is order null, hence by Lemma \ref{oso} any order bounded disjoint countably indexed non-terminating net is $\sigma$-order null. Let $0\notin A\subseteq X_{+}$ be disjoint and order bounded. Then, $A\cup\left\{0\right\}$ with $x_{0}=0$ satisfy the conditions of Lemma \ref{countcard}, and so $A$ is countable.  Thus $X$ has the CSP which, together with $\sigma$-order completeness, yields order completeness, see \cite[Theorem 23.6]{Luxemburg:71} or \cite[Corollary 2.11]{Deng:25}.
\end{proof}

We can now relate strong first countability to $\sigma$-strong first countability.

\begin{thm}\label{main1}The following are equivalent. \begin{enumerate}[(i)]
\item $X$ is strongly first countable.
\item $X$ is $\sigma$-strongly first countable and has the CSP.
\item $X$ is $\sigma$-strongly first countable and almost $\sigma$-order complete.
\item $X$ is almost $\sigma$-order complete and satisfies \eqref{so-sfc-CD}.
\item $X^{\delta}$ is $\sigma$-strongly first countable.
\item $X^{\delta}$ is strongly first countable.
\end{enumerate}
\end{thm}
\begin{proof}
(i)$\Leftrightarrow$(ii) was proven in Corollary \ref{main0}, (i)$\Leftrightarrow$(vi) follows from part (i) of Proposition \ref{dense}. (ii)$\Rightarrow$(iii) follows from the fact that the CSP implies almost $\sigma$-order completeness, see \cite[Lemma 1.44]{Aliprantis:03}. (v)$\Rightarrow$(vi) follows from Proposition \ref{so-sfc} and Corollary \ref{main0}. (iv)$\Rightarrow$(iii) follows from the fact that \eqref{so-sfc-CD} implies $\sigma$-strong first countability.\medskip

(iii)$\Rightarrow$(iv),(v): Assume that $X$ embeds as a super order dense sublattice in a $\sigma$-order complete vector lattice $Y$. By replacing $Y$ with $I\left(X\right)$ if needed we may assume that $X$ is majorizing in $Y$. Hence, according to part (ii) of Proposition \ref{dense} it follows that $Y$ is $\sigma$-strongly first countable. Proposition \ref{so-sfc} yields that $Y$ is order complete and so $X^{\delta}=Y$ is $\sigma$-strongly first countable. In particular, $Y$ satisfies  \eqref{so-sfc-CD}, and so according to Proposition \ref{odmstar}, $X$ satisfies \eqref{so-sfc-CD} as well.
\end{proof}

As is shown in Example \ref{bnn}, $\sigma$-strong first countability does not imply strong first countability.  However, it follows immediately from Theorem \ref{main1} that these notions are equivalent for order complete spaces.

\begin{cor}
Let $X$ be almost $\sigma$-order complete.  Then the following are equivalent. \begin{enumerate}[(i)]
    \item $X$ is strongly first countable.
    \item $X$ is $\sigma$-strongly first countable.
    \item $X$ satisfies \eqref{so-sfc-CD}.
\end{enumerate}
\end{cor}

We illustrate the usefulness of the condition \eqref{so-sfc-CD} by considering some concrete spaces.  Recall that $c$ stands for the sublattice of $\ell_{\8}$ consisting of all convergent real sequences.

\begin{prop}\label{so-clinfty}
$c$ and $\ell_{\8}$ are $\sigma$-strongly first countable, hence also strongly first countable.
\end{prop}

\begin{proof}
We prove the result for $c$; the proof for $\ell_{\8}$ is similar. For each $n\in\mathbb N$, set
  \begin{displaymath}
    v_n=\bigl(\overbrace{\tfrac1n,\dots,\tfrac1n}^{n\mbox{ times}},
      n,n,\dots\bigr).
  \end{displaymath}
Let $C=\{v_n\mid n\in\mathbb N\}$. Suppose that $x_k\downarrow 0$. Take $D=\bigl\{v_n\mid n\ge\norm{x_1}\bigr\}$. Then $\bigwedge D=0$, and for every $d\in D$ there exists $k\in\mathbb N$ such that $x_k\le d$.  By Theorem \ref{sigmastrong}, $c$ is $\sigma$-strongly first countable.  Since $c$ has the CSP, it follows from Corollary \ref{main0} that it is strongly first countable.
\end{proof}

This result may be generalized.  A collection of non-empty open subsets of a topological space is said to be a \term{$\pi$-base} if every non-empty open set contains a member of the collection. We also say that a topological space satisfies the \term{countable chain condition (CCC)} if every disjoint collection of open sets is countable.  It is well known, see for instance \cite{Kandic:18} or \cite[Theorem 2.37]{Deng:25}, that if $K$ is compact, then it satisfies the CCC if and only if $C(K)$ has the CSP.  It is also easy to see that the existence of a countable $\pi$-base implies the CCC.  We present a simple proof that the existence of a countable $\pi$-base implies strong first countability; for the converse see Corollary \ref{corrr}.

\begin{prop}\label{so-cPb}
  Let $K$ be a compact Hausdorff space with a countable $\pi$-base.  Then $C(K)$ is $\sigma$-strongly first countable and strongly first countable.
\end{prop}

\begin{proof}
It follows from the discussion above that $C(K)$ has the CSP so that ${\rm o}$ and $\sigma{\rm o}$ coincide on $C(K)$.  It therefore suffices to show that condition \eqref{so-sfc-CD} is satisfied. Let $\{U_{n} \}_{n\in\N}$ be a countable $\pi$-base in $K$. For each $n\in \N$, pick a point $t_n\in U_n$. For every pair $m,n\in\mathbb N$, we use Urysohn's Lemma to find $f_{m,n}\in C(K)$ such that $f_{m,n}(t_n)=\frac1m$, $\frac1m\1\le f_{m,n}\le m\1$, and $f_{m,n}(t)=m$ for all $t\notin U_n$.  Let $C\defeq \{f_{m,n}\}_{m,n\in\mathbb N}$. Suppose that $g_k\downarrow 0$ in $C(K)$. Let $D$ be the set of all $f\in C$ such that $f\ge g_k$ for some $k\in \N$.\medskip

Fix a non-empty open set $U$ in $K$ and $m\in\mathbb N$ with $m\ge\norm{g_1}$. Since $g_k\downarrow 0$, \cite[Lemma 3.1]{Bilokopytov:22} asserts that there exist a non-empty open set $V$ and $k\in\mathbb N$ such that $V\subseteq U$ and $g_k(t)\le\frac1m$ for all $t\in V$. Find $n$ such that $U_n\subseteq V$. Then $g_k\le f_{m,n}$ because if $t\in U_n$ then $g_k(t)\le\frac1m\le f_{m,n}(t)$ and if $t\notin U_n$ then $g_k(t)\le\norm{g_1}\le m\le f_{m,n}(t)$. It follows that $f_{m,n}\in D$. Furthermore, since $t_n\in U$ and $f_{m,n}(t_n)\le\frac1m$, we conclude by  \cite[Lemma~3.1]{Bilokopytov:22} that $\bigwedge D=0$.
\end{proof}

Note that $[0,1]$ has a countable $\pi$-base, hence $\sigma$-order convergence on $C[0,1]$ is strongly first countable. More generally, if $K$ is metrizable, it has a countable base, hence a countable $\pi$-base. Furthermore, both $c$ and $\ell_{\8}$ may be represented as $C(K)$ where $K$ has a countable $\pi$-base, so Proposition~\ref{so-clinfty} is a special case of Proposition~\ref{so-cPb}.

We now present an example of a compact Hausdorff space $K$ so that $C(K)$ is $\sigma$-strongly first countable, but not strongly first countable.

\begin{ex}\label{bnn}
Let $K$ be a compact, Hausdorff \term{almost P space}, a.k.a. a \term{P' space}; that is, every closed and nowhere dense $G_{\delta}$ set in $K$ is empty, see \cite{Veksler:73}. As is shown in \cite[Theorem 2]{Veksler:73}, this condition is equivalent to the fact that $\sigma$-order convergence coincides with norm convergence on $C\left(K\right)$.  Taking $C:=\left\{\frac{1}{m}\1,~ m\in\N\right\}$ we see that $C\left(K\right)$ satisfies the condition \eqref{so-sfc-CD}, and so is $\sigma$-strongly first countable.

According to \cite{Veksler:73}, $\beta\N\backslash \N$ is an almost P space, hence $C\left(\beta\N\backslash \N\right)$ satisfies \eqref{so-sfc-CD} and is therefore $\sigma$-strongly first countable.  On the other hand, $C\left(\beta\N\backslash \N\right)$ does not have the CCC, since disjoint clopen sets in $\beta\N\backslash \N$ correspond to almost disjoint families in $\N$, which can be uncountable, see \cite[p. 80]{Koppelberg:89}. Therefore, \eqref{so-sfc-CD} does not imply the CSP, and in particular strong first countability.
\qed\end{ex}

\subsection{The case of Boolean algebras}

In this section we consider order- and $\sigma$-order convergence on Boolean algebras.  Transitioning from vector lattices to Boolean algebras allows us to better see the combinatorial aspects of the properties under consideration, and in particular to ``distill'' the opposite properties.

Throughout this section we denote by $B$ a Boolean algebra.  For basic notation, we follow \cite[Chapter 1]{Koppelberg:89}.  The least and greatest elements of $B$ are denoted $0$ and $1$, respectively.  For $b\in B$, we write $1-b$ for the \term{complement} of $b$ in $B$; that is, $1-b$ is the unique element of $B$ so that $b\vee (1-b)=1$ and $b\wedge (1-b)=0$. More generally, if $a\wedge b=0$ we say that $a$ and $b$ are \term{disjoint}, and write $a\perp b$.

Observe that for every $b>0$ in $B$, the order interval $[0,b]$ is a Boolean algebra with respect to the order it inherits from $B$.  We call $b>0$ an \term{atom} if $[0,b]=\{0,b\}$.  If $a\le b$, we denote by $b-a$ the complement of $a$ in the Boolean algebra $[0,b]$; so $(b-a)\vee a=b$ and $(b-a)\wedge a=0$.  More generally, for $0<b$ and $a\in B$, we denote by $(b-a)^+$ the complement of $a\wedge b$ in $[0,b]$.

The definitions of ($\sigma$-)order convergence in Boolean algebras are analogous to the ones for vector lattices.  As for vector lattices, we call $B$ ($\sigma$-)strongly first countable if ($\sigma$-)order convergence on $B$ is strongly first countable.  We will take advantage of the fact that Theorem \ref{sigmastrong}, Proposition \ref{so-sfc} and Theorem \ref{main1} remain valid for Boolean algebras, as well as the fact that a principal ideal in a ($\sigma$-)strongly first countable Boolean algebra is ($\sigma$-)strongly first countable.

We call $A\subseteq B\backslash\left\{0\right\}$ a \term{$\pi$-base} if for every $b>0$ there is $a\in A\cap\left(0,b\right]$.  Analogously, a $\pi$-base of a vector lattice $X$ is  set $A\subseteq X_+$ so that $A\cap (0,x]\neq \varnothing$ for every nonzero $x\in X_+$.  Let us prove the analogue of Proposition \ref{so-cPb}.

\begin{prop}\label{ccpb}
If $B$ has a countable $\pi$-base, then it is strongly first countable, hence also $\sigma$-strongly first countable.
\end{prop}
\begin{proof}
Since $B$ has a countable $\pi$-base it satisfies the CSP, and so it suffices to show that condition \eqref{so-sfc-CD} holds. Let $\left\{b_{m}\right\}_{m\in\N}\subseteq B\backslash\left\{0\right\}$ be a $\pi$-base for $B$.  We show that $C:=\left\{1-b_{m}\right\}_{m\in\N}$ has the required property. Assume that $x_{n}\downarrow 0$, and let $D=\left\{1-b_{m} \mid 1-b_{m}\ge x_{n} \text{ for some } n\in\N\right\}$. If $\bigwedge D\ne 0$, then there is $b>0$ such that $d\ge b$ for all $d\in D$. Since $\bigvee\limits_{n\in\N}\left(1- x_{n}\right)=1$, there is $n\in\N$ with $\left(1-x_{n}\right)\wedge  b>0$. It follows that there is $m\in\N$ such that $\left(1-x_{n}\right)\wedge  b\ge b_{m}$, in particular, $1-x_{n}\ge b_{m}$, hence $1-b_{m}\in D$ so that $1-b_m\ge b$.  We have $b,1-b\le 1-b_{m}$, hence $b_{m}=0$, a contradiction.
\end{proof}

Let $C=\{c_\alpha\}_{\alpha \in A}$ be a subset of $B\setminus \{0\}$.  According to \cite[Def. 0.1]{Balcar:89}, a \term{disjoint refinement} of $C$ is a collection $D=\{d_\alpha\}_{\alpha \in A}$ of mutually disjoint elements of $B\setminus\{0\}$ so that $d_\alpha \le c_\alpha$ for every $\alpha \in A$.  For a cardinal number $\kappa$, we say that $B$ satisfies the disjoint refinement property for systems of cardinality at most $\kappa$, in short, $B$ has ${\rm Rp}(\kappa)$, if every $C=\{c_\alpha\}_{\alpha \in \kappa}$ has a disjoint refinement, see \cite[Def. 1.0]{Balcar:89}.  We will be interested exclusively in $B$ satisfying ${\rm Rp}(\omega)$.  The relationship between the existence of a countable $\pi$-base and ${\rm Rp}(\omega)$ is given in the following result.

\begin{prop}\label{cdp}
$B$ satisfies ${\rm Rp}(\omega)$ if and only if for every $b>0$ the Boolean algebra $\left[0,b\right]$ does not have a countable $\pi$-base.
\end{prop}
\begin{proof}

Assume that $B$ satisfies ${\rm Rp}(\omega)$.  Applying ${\rm Rp}(\omega)$ to any constant sequence in $B$ we see that $B$ has no atoms.  Suppose that $b>0$ is such that $\left[0,b\right]$ has a countable $\pi$-base $\left\{b_{n}\right\}_{n\in\N}\subseteq B\backslash\left\{0\right\}$. Let $\left\{a_{n}\right\}_{n\in\N}\subseteq B\backslash\left\{0\right\}$ be a disjoint refinement of $\left\{b_{n}\right\}_{n\in\N}$.   In particular, $a_{n}\le b_{n}$ for every $n\in\N$.  Since $a_{n}$ is not an atom, we can ``shrink'' it if needed and assume that $a_{n}<b_{n}$ for every $n\in\N$.  Since $\left\{b_{n}\right\}_{n\in\N}$ is a $\pi$-base for $[0,b]$, for every $n\in\N$ there exists $m_n\in\N$ such that $b_{m_n}\le a_{n}$.  As $a_{n}< b_{n}$, it follows that $m_n\ne n$. Thus, $a_{m_n}\le b_{m_n}\le a_{n}$, which contradicts disjointness of $a_{n}$ and $a_{m_n}$.\medskip

Assume that for every $b>0$, $\left[0,b\right]$ does not have a countable $\pi$-base.  Note that, in particular, no $b>0$ is an atom.  Let $\left\{b_{n}\right\}_{n\in\N}\subseteq B\backslash\left\{0\right\}$.

\textbf{Claim. }There is $0<a\le b_1$ such that $b_m\not\le a_1$ for every $m\in\N$. Let $M\defeq \left\{ m\neq 1 \mid b_m\wedge b_1>0\right\}$.  If $M=\varnothing$, i.e. $b_m\wedge b_1 = 0$ for every $m\neq 1$, take any $0<a<b_1$.  Such an $a$ exists since $b_1$ is not an atom.  Clearly, in this case $0<a\le b_1$ and $b_m \not\le a$ for all $m\in\N$.  Suppose $M\neq \varnothing$.  Since $\{b_m\wedge b_1\}_{m\in M}$ is not a $\pi$-base for $[0,b_1]$ there exists $0<a< b_1$ so that for all $m\in M$, $b_m\wedge b_1 \not\le a$.  If $m\notin M$, i.e. if $b_m\wedge b_1=0$, then since $a\le b_1$ we have $a\wedge b_m =0$ so that $b_m\not \le a$.  If $m\in M$ then $b_m\le a$ implies $b_m\wedge b_1 \le a\wedge b_1=a$, a contradiction, so also in this case $b_m\not\le a$.\medskip

We will now inductively construct a disjoint sequence $\left\{a_{n}\right\}_{n\in\N}\subseteq B\backslash\left\{0\right\}$ such that for all $m,n\in\N$, $a_{n}\le b_{n}$ and $b_{m}\not\le \bigvee\limits_{j=1}^{n}a_{j}$. Existence of $a_1$ follows immediately from the claim. Fix $k\in \N$ and assume that $a_{1},...,a_{k-1}$ have been constructed. Let $B':=\left[0,1- \bigvee\limits_{j=1}^{k-1} a_{j}\right]$, and $b_n'\defeq \left(b_n-\bigvee\limits_{j=1}^{k-1} a_{j}\right)^{+}$, for $n\in \N$. By the inductive hypothesis, we have that $\left\{b'_{n}\right\}_{n\in\N}\subseteq B'\backslash\left\{0\right\}$, and according to the claim, there is $0<a_k\le b'_k$ such that $b'_m\not\le a_k$, for every $m\in\N$. The last condition implies $b_{m}\not\le \bigvee\limits_{j=1}^{k}a_{j}$, for every $m\in\N$.  Since $a_k\le b_k'$ and $b_k'\perp \bigvee\limits_{j=1}^{k-1} a_{j}$ it follows that $a_k \perp a_j$ for all $j=1,\ldots,k-1$. 
\end{proof}

\begin{prop}\label{cdp1}
If $B$ is $\sigma$-order complete and satisfies ${\rm Rp}(\omega)$, then it is not $\sigma$-strongly first countable.
\end{prop}
\begin{proof}
Assume the contrary.  According to Proposition \ref{so-sfc} there exists $C=\left\{c_{n}\right\}_{n\in\N}\subseteq B\backslash\left\{0,1\right\}$ such that for any $x_{n}\downarrow 0$ there exists $D\subseteq C$ so that $\bigwedge D=0$, and, for every $d\in D$ there exists $k\in\N$ with $x_{k}\le d$.

For every $n\in\N$ let $b_{n}\defeq 1-c_{n}>0$ and let $\left\{a_{n}\right\}_{n\in\N}\subseteq B\backslash\left\{0\right\}$ be disjoint and such that $a_{n}\le b_{n}$, for every $n\in\N$.  For every $n\in\N$ let $\left\{a_{n}^{m}\right\}_{m\in\N}\subseteq B\backslash\left\{0\right\}$ be disjoint and such that $a_{n}^{m}\le a_{n}$, for every $m\in\N$.  For every $m\in\N$ let $x_{m}:=\bigvee\limits_{n\in\N}a_{n}^{m}$.  Then $\left(x_{m}\right)_{m\in\N}$ is a disjoint sequence; defining $y_{m}:=\bigvee\limits_{k\ge m}x_{k}$ we get $y_{m}\downarrow 0$. On the other hand $y_{m}\ge a_{n}^{m}\bot c_{n}$, hence $y_{m}\not\le c_{n}$ for all $m,n\in\N$.  It is easy to see that we cannot find $D\subseteq C$ with the required properties, which leads to a contradiction.
\end{proof}

\begin{ex}\label{measure}
Any atomless measure algebra $\left(B,\mu\right)$ is order complete, has the CSP and satisfies ${\rm Rp}(\omega)$, see \cite[Example 1.5]{Balcar:89}.
\qed\end{ex}

We are now ready to characterize strong first countability for Boolean algebras.

\begin{thm}\label{cso-cPb}
$B$ is strongly first countable if and only if it has a countable $\pi$-base.
\end{thm}
\begin{proof}
Sufficiency was proven in Proposition \ref{ccpb}.  Suppose that $B$ is strongly first countable.  According to Theorem \ref{main1} and Proposition \ref{so-sfc} we may assume that $B$ is complete and has the CSP.  Let $A$ be the collection of $b>0$ in $B$ such that $\left[0,b\right]$ has a countable $\pi$-base. Let us show that $\bigvee A=1$. Assume that this is not true.  Then there exists $0<c\in B$ so that $c\bot a$ for every $a\in A$, for instance $c\defeq 1-\bigvee A$.  Clearly, if $0<d\le c$, then $\left[0,d\right]$ does not have a countable $\pi$-base.  Hence, according to Proposition \ref{cdp}, $\left[0,c\right]$ satisfies ${\rm Rp}(\omega)$, and since it is $\sigma$-order complete, by Proposition \ref{cdp1}, it is not $\sigma$-strongly first countable, hence not strongly first countable, a contradiction.\medskip

As $B$ has the CSP, there exists $\left\{a_{n}\right\}_{n\in\N}\subseteq A$ such that $\bigvee\limits_{n\in\N} a_{n}=1$.  For every $n\in \N$, let $C_{n}$ be a countable $\pi$-base for $\left[0,a_{n}\right]$.  We claim that $\bigcup\limits_{n\in\N}C_{n}$ is a (countable) $\pi$-base for $B$.  Indeed, for any $b>0$ there is $n\in\N$ such that $0<b\wedge a_{n}\le a_{n}$, and so there is $c\in C_{n}$ such that $c\le b\wedge a_{n}\le b$.
\end{proof}

\begin{rem}\label{gleason}
Assume that $B$ is complete and has a countable $\pi$-base.  It turns out that these conditions leave little room for what $B$ can be.  In particular, it is easy to see that if $B$ is atomic, then $B$ is isomorphic to $2^{n}$, where $n\in\N\cup\left\{0,\8\right\}$.

If $B$ is atomless, it is uniquely determined.  Indeed, let $A\subseteq B$ be a countable $\pi$-base, and let $B\left(A\right)$ be the Boolean subalgebra generated by $A$ in $B$.  Then, $B\left(A\right)$ is countable and atomless, and hence isomorphic to $Clop\left(K\right)$, the algebra of clopen subsets of Cantor's space $K$, see \cite[Corollary 5.16 and Example 7.24]{Koppelberg:89}.  Note that $B$ is an order complete Boolean algebra which contains $B\left(A\right)$ as an order dense subalgebra. It follows that $B$ is the order completion of $B\left(A\right)^{\delta}$, and so $B$ is isomorphic to  $RO\left(K\right)$, the algebra of regular open sets in $K$, which in turn is isomorphic to $RO\left(\left[0,1\right]\right)$, see \cite[Proposition 7.17]{Koppelberg:89}.
\qed\end{rem}

\subsection{A characterization of strong first countability}\label{Section:  The main characterisation}

We now return to vector lattices.  The following construction provides a bridge from the case of Boolean algebras to the vector lattice case. Let $K$ be the Stone space of a Boolean algebra $B$.  We may assume that $B$ is the Boolean algebra of clopen subsets of $K$.  Let $X_B\subseteq C(K)$ be the sublattice of simple functions over $B$; that is, $X_B$ is the linear span of all indicator functions of clopen subsets of $K$.  Since the clopen sets in $K$ form a basis $K$, it follows form the Stone-Weierstrass theorem that $X_B$ is norm-dense in $C(K)$.  According to Lemma \ref{cor}, $X_B$ is super order dense in $C(K)$, and so by Proposition \ref{dense} the inclusion of $X_B$ into $C(K)$ is an $\mathrm{o}$- and $\sigma\mathrm{o}$-embedding.  Furthermore, $B$ is included into $X_B$ via characteristic functions of clopen sets.

\begin{lem}\label{bavl}
The inclusion of $B$ into $X_B$ introduced above is an embedding with respect to both order- and $\sigma$-order convergence.
\end{lem}
\begin{proof}
As usual, we only provide a proof for $\sigma$-order convergence. If $B\ni U_{n}\downarrow \varnothing$, then $\bigcap\limits_{n\in\N}U_{n}$ is nowhere dense, and it is easy to see that $\bigwedge\limits_{n\in\N}\1_{U_{n}}=\0$ in $X_B$. Arguing as in the proof of Proposition \ref{dense} we see that the embedding is $\sigma$-order continuous. Assume that $\left(U_{\alpha}\right)\subseteq B$ is such that $\1_{U_{\alpha}}\goesso\1_{U}$, where $U\in B$. Let $\left(f_{n}\right)\subseteq X_B$ be a dominating sequence for this convergence. Let $V_{n}:=f^{-1}_n\left[\left[1,+\8\right)\right]\in B$. We have that $V_{n}\downarrow \varnothing$. For every $n\in\N$ there is $\alpha_{n}$ such that $\1_{U_{\alpha}\triangle U}=\left|\1_{U_{\alpha}}-\1_{U}\right|\le f_{n}$, hence $U_{\alpha}\triangle U\subseteq V_{n}$, for every $\alpha\ge\alpha_{n}$. It follows that $U_{\alpha}\goesso U$ in $B$.
\end{proof}

In order to formulate the main result of this section we recall some terminology and notation.  The collection $\Bo_{X}$ of all bands in $X$ is a complete Boolean algebra, see for instance \cite[page 326]{Luxemburg:71}.  A vector  $x\in X_{+}$ is a \emph{component} of $y\in X_{+}$ if $x\wedge \left(y-x\right)=0$.  Note that components of $y$ form a Boolean algebra (this follows from \cite[Theorem~1.49]{Aliprantis:06}), denoted by $C_{y}$.  We are now ready for the main result of the section.

\begin{thm}\label{main2}Let $X$ be an Archimedean vector lattice. The following are equivalent:
\item[(i)] $X$ is strongly first countable;
\item[(ii)] There exist $\left\{a_{n}\right\}_{n\in\N},\left\{b_{n}\right\}_{n\in\N}\subseteq X_+$ such that $X_{+}\backslash\left\{0\right\}=\bigcup\limits_{n\in\N} \left[a_{n},b_{n}\right]$;
\item[(iii)] There exists an increasing sequence $(v_n)$ in $X_+$ such that $X=I_{\{v_n\}}$ and $\Bo_{X}$ has a countable $\pi$-base;
\item[(iv)] There exist $\left\{a_{n}\right\}_{n\in\N}, \left\{b_{n}\right\}_{n\in\N}\subseteq X^{\delta}_+$ such that $X^{\delta}_{+}\backslash\left\{0\right\}=\bigcup\limits_{n\in\N} \left[a_{n},b_{n}\right]$;
\item[(v)] There exists an increasing sequence $(v_n)$ in $X_{+}^{\delta}$ such that $X^\delta=I_{\{v_n\}}$, and $C_{v_{n}}$ (in $X^{\delta}$) has a countable $\pi$-base for every $n\in\N$.
\end{thm}
\begin{proof}
(i)$\Rightarrow$(v):  Suppose that $X$ is strongly first countable.  According to Theorem \ref{main1}, $X^\delta$ is strongly first countable.  Therefore, by Proposition \ref{sos1c}, there exists an increasing sequence $(v_n)$ in $X^\delta_+$ so that $X^\delta=I_{\{v_n\}}$ and $I_{v_n}$ is strongly first countable for every $n\in \N$.  Fix $n\in \N$.  Since $X^\delta$ is order complete, $I_{v_n}$ is isomorphic to $C(K)$, where $K$ is the Stone space of $C_{v_n}$.  Since $I_{v_n}$ is strongly first countable, $C_{v_n}$ is strongly first countable according to Lemma \ref{bavl} and the preceding discussion, and so has a countable $\pi$-base by Theorem \ref{cso-cPb}.

(v)$\Rightarrow$(i): By Theorem \ref{main1} it is enough to show $X^\delta$ is strongly first countable.  To see that this is so, it suffices by Proposition \ref{sos1c} to show that each $I_{v_n}$ is strongly first countable.  As explained above, $I_{v_n}$ is isomorphic to $C(K)$ with $K$ the Stone space of $C_{v_n}$.  Since elements of $C_{v_n}$ correspond to the clopen subsets of $K$, which form a base for $K$, it follows that $K$ has a countable $\pi$-base, and so $I_{v}$ is strongly first countable by Proposition \ref{so-cPb}.\medskip

(v)$\Rightarrow$(iv):  By Freudenthal's Spectral Theorem (\cite[Theorem 1.58]{Aliprantis:03}), for every $n\in\N$ we have $\left(0,v_{n}\right]=\bigcup\limits_{k\in\N,~ c\in C_{v_{n}}}\left[\frac{1}{k}c, v_{n}\right]$.  For $n\in\N$ let $\left\{v_{nm}\right\}_{m\in\N}$ be a $\pi$-base for $C_{v_{n}}$. It follows that $\left(0,v_{n}\right]=\bigcup\limits_{m,k\in\N}\left[\frac{1}{k}v_{nm}, v_{n}\right]$, from where $X^\delta_{+}\backslash\left\{0\right\}=\bigcup\limits_{m,k,n\in\N}\left[\frac{1}{k}v_{nm}, kv_{n}\right]$.\medskip

(iv)$\Leftrightarrow$(ii) follows from the fact that $X$ is order dense and majorizing in $X^{\delta}$, and Proposition \ref{dense}.\medskip

(ii)$\Rightarrow$(iii): For every $n\in\N$, let $v_{n}:=b_{n}$; clearly, $X=I_{\{v_n\}}$.  We show that $\left\{a_{n}^{dd}\right\}_{n\in\N}$ is a $\pi$-base for $\Bo_{X}$.  Consider $B\in \Bo_{X}\setminus \{\{0\}\}$ and some $0<x\in B$.  There exists $n\in\N$ so that $0<a_n\le x$, hence $a_n\in B$.  Since $a_{n}^{dd}$ is the smallest band containing $a_n$, it follows that $a_n^{dd}\subseteq B$.  \medskip

(iii)$\Rightarrow$(v): Since $X$ is majorizing in $X^{\delta}$ it follows that the ideal generated by $\{v_n\}_{n\in\N}$ in $X^{\delta}$ is all of $X^\delta$.  Let $\left\{H_{m}\right\}_{n\in\N}$ be a countable $\pi$-base of $\Bo_{X}$.  It is easy to see that, for every $n\in\N$,  $\left\{P_{H_{m}}v_{n}\right\}_{m\in\N}$ is a $\pi$-base for $C_{v_{n}}$.
\end{proof}

\begin{cor}\label{corrr}For a compact Hausdorff space $K$, the following are equivalent:
\item[(i)] $K$ has a countable $\pi$-base;
\item[(ii)] The complete Boolean algebra $RO(K)$ of regular open subsets of $K$ has a countable $\pi$-base;
\item[(iii)] $C(K)$ has a countable $\pi$-base;
\item[(iv)] $C(K)$ is strongly first countable.
\end{cor}
\begin{proof}
(i)$\Leftrightarrow$(ii) follows from the fact that every open set contains a regular open set. (ii)$\Leftrightarrow$(iv) follows from Theorem \ref{main2} and the fact that there is a Boolean isomorphism between $\Bo_{C(K)}$ and $RO(K)$ (see \cite[Exercise 22.10]{Luxemburg:71}). It is also straightforward to prove (i)$\Leftrightarrow$(iii) using the condition (ii) in Theorem \ref{main2}.
\end{proof}

\begin{rem}
It follows from Remark \ref{rem-strong} that if $X$ admits a complete metrizable locally solid topology, then it is strongly first countable if and only if it is order isomorphic to $C(K)$ where $K$ is a compact Hausdorff space with a coutable $\pi$-base.\qed
\end{rem}

\begin{cor}\label{bvl}
Let $K$ be the Stone space of a Boolean algebra $B$. Then $B$ is strongly first countable if and only if $C(K)$ is.
\end{cor}

Using part (i) of Proposition \ref{dense} we obtain the following result.

\begin{cor}\label{cbvl}
Let $K$ be compact Hausdorff and let $X\subseteq\Co\left(K\right)$ be a norm-dense sublattice. Then $X$ is strongly first countable if and only if $K$ has a countable $\pi$-base.
\end{cor}

\begin{rem}
It is also true that in the context of Corollary \ref{bvl}, $B$ is $\sigma$-strongly first countable if and only if $C(K)$ is. The converse implication is a consequence of Lemma \ref{bavl}, but the proof of the direct implication is rather long and technical. The main idea is to follow the scheme of the proof of Proposition \ref{dense}. Namely, start with countable fixed intersection closed local basis $\mathcal{H}$ and show that $\Ho':=\left\{H^{m} \mid H\in\Ho,~ m\in\N\right\}$ witnesses strong first countability of $\sigma$-order convergence on $X$, where $H^{m}:=\left\{f\in X\cap\left[-\1,\1\right],~ \supp\left(\left|f\right|-\frac{1}{m}\1\right)^{+}\in H\right\}$.\medskip

Note that the proof of the statement, as well as of Proposition \ref{fsc-unbdd}, could be greatly simplified if we had a developed theory of ``local'' convergence, i.e. convergence to a fixed point of the space, as opposed to convergence to every point. For example, in the proof of Proposition \ref{bavl} we essentially observe that convergence to $\0$ in $X$ is initial with respect to maps $f\mapsto \supp\left(\left|f\right|-\frac{1}{m}\1\right)^{+}$ into $B$, which act like ``pseudo-norms''. Hence, first countability of $B$ implies first countability of $X$ at $\0$, thus everywhere.
\qed\end{rem}

\begin{cor}\label{corm}
If $\mu$ is a finite measure, then $L_{\8}\left(\mu\right)$ is $\sigma$-strongly first countable if and only if $L_{\8}\left(\mu\right)$ is isomorphic to $\ell_{\8}$, or finite-dimensional.
\end{cor}
\begin{proof}
Sufficiency was proven in Proposition \ref{so-clinfty}. Since $L_{\8}\left(\mu\right)$ has the CSP, in order to show necessity we may assume that $L_{\8}\left(\mu\right)$ is strongly first countable.  The atomless part of $L_{\8}\left(\mu\right)$ is a principal ideal, and so to reach a contradiction we may assume that $\mu$ is atomless to begin with.  Hence, $L_{\8}\left(\mu\right)$ is isomorphic to $C(K)$, where $K$ is the Stone space of the measure algebra $B$ of $\mu$.  Note that $RO(K)\simeq B$, and by Example \ref{measure} this algebra does not have a countable $\pi$-base which, according to Corollary \ref{corrr}, implies that $L_{\8}\left(\mu\right)$ is not strongly first countable.
\end{proof}

\subsection{The case of uo convergence}

In this subsection we consider (strong) first countability of the unbounded modifications $\mathrm{uo}$ and $\mathrm{u\sigma o}$ of $\mathrm{o}$ and $\mathrm{\sigma o}$, respectively.

\begin{prop}\label{ucsp}Let $X$ be an Archimedean vector lattice. The following are equivalent:
\item[(i)] $X$ has countable supremum property;
\item[(ii)] $\mathrm{uo}=\mathrm{u\sigma o}$ on $X$;
\item[(iii)] $\mathrm{uo}$ is Frechet-Urysohn on $X$;
\item[(iv)] $\mathrm{uo}$ is sequential on $X$.
\end{prop}
\begin{proof}
(i)$\Rightarrow$(ii) follows from Proposition \ref{csp}, and (iii)$\Rightarrow$(iv) is trivial. (ii)$\Rightarrow$(i) and (iv)$\Rightarrow$(i) follow from Remark \ref{bcsp} and the fact that $\mathrm{uo}$ and $\mathrm{u\sigma o}$ agree with $\mathrm{o}$ and $\sigma\mathrm{o}$, respectively, on order intervals. (i)$\Rightarrow$(iii) is proven in \cite[Theorem 3.15]{Deng:25}.
\end{proof}

Recall that $X^u$ denotes the universal completion of $X$, see for instance \cite[Section 7.2]{Aliprantis:03}.

\begin{prop}[cf. \cite{Taylor:18}, Theorem 6.2]\label{ufc}Let $X$ be an Archimedean vector lattice. The following are equivalent:
\item[(i)] $X$ has the countable supremum property and there is a countable set $C\subseteq X$ such that $X=C^{dd}$;
\item[(ii)] Disjoint sets in $X$ are at most countable;
\item[(iii)] $\Bo_{X}$ has the countable supremum property;
\item[(iv)] $X^{u}$ has countable supremum property;
\item[(v)] $\mathrm{uo}$ is first countable on $X$;
\item[(vi)] $\mathrm{uo}$ is weakly first countable on $X$;
\item[(vii)] $\mathrm{uo}$ is first countable on $X^{u}$;
\item[(viii)] $\mathrm{uo}$ is sequential on $X^{u}$.
\end{prop}
\begin{proof}
The equivalence of (i)-(iv) and (vi) see in \cite[theorems 2.26 and 3.17]{Deng:25}. (iv)$\Leftrightarrow$(viii) follows from Proposition \ref{ucsp}. (i)$\Rightarrow$(v) follows from combining Propositions \ref{csp} and \ref{fsc-unbdd} with \cite[Lemma 2.2]{Li:18}. Since first countability implies weak first countability, we get (v)$\Rightarrow$(vi). Applying equivalence of (iv) and (v) to $X^{u}$ yields (iv)$\Leftrightarrow$(vii).
\end{proof}

\begin{prop}\label{uso}
If $X$ is ($\sigma$-)strongly first countable, then $\mathrm{uo}$ ($\mathrm{u\sigma o}$) is strongly first countable on $X$. Conversely, if $\mathrm{uo}$ ($\mathrm{u\sigma o}$) is strongly first countable on $X$, then every principal ideal of $X$ is ($\sigma$-)strongly first countable.
\end{prop}
\begin{proof}
The first claim is a combination of Propositions \ref{sos1c} and \ref{fsc-unbdd}.

Let us prove the second claim (as usual only for the ``$\sigma$'' case).  Let $u\ge 0$ and note that $\mathrm{u\sigma o}_X$ and $\sigma\mathrm{o}_{I_u}$ agree on $[-u,u]$. Following the proof of Corollary \ref{so-Iu-sfc} we conclude that $I_{u}$ is $\sigma$-strongly first countable.
\end{proof}

It is often the case that properties of uo convergence have nicer characterizations than properties of order convergence. Strong first countability is not an exception.

\begin{thm}\label{uo}Let $X$ be an Archimedean vector lattice. The following are equivalent:
\item[(i)] $\mathrm{uo}_{X}$ is strongly first countable;
\item[(ii)] $X$ is almost $\sigma$-order complete and $\mathrm{u\sigma o}_{X}$ is strongly first countable;
\item[(iii)] Every principal ideal is strongly first countable and there is a countable set $A\subseteq X$ such that $X=A^{dd}$;
\item[(iv)] There exists an increasing sequence $(v_n)$ in $X_+$ such that $X=\{v_n\}^{dd}$ and $I_{v_n}$ is strongly first countable for every $n\in \N$, i.e. it is isomorphic to a dense sublattice of $\Co\left(K_{v_{n}}\right)$, where $K_{v_{n}}$ has a countable $\pi$-base;
\item[(v)] $X$ has a countable $\pi$-base;
\item[(vi)] $\Bo_{X}$ has a countable $\pi$-base;
\item[(vii)] $\mathrm{uo}_{X^{u}}$ is strongly first countable;
\item[(viii)] There exists $n\in \N\cup\left\{0,\8\right\}$ so that $X^{u}\simeq \Co_{\8}\left(K\right)\oplus \R^{n}$, where $K=\varnothing$ or $K$ is the Stone space of $RO([0,1])$.
\end{thm}

\begin{proof}
First, (iii)$\Rightarrow$(iv) and (v)$\Rightarrow$(vi) are straightforward. (i)$\Rightarrow$(iii) follows from Propositions \ref{ufc} and \ref{uso}. (vi)$\Leftrightarrow$(viii) follows from Remark \ref{gleason}. (vii)$\Rightarrow$(i) follows from the fact that the embedding of $X$ into $X^{u}$ is a $\mathrm{uo}$-embedding (see \cite[Theorem 3.2]{Gao:17}).\medskip

(i)$\Rightarrow$(ii): It follows from Proposition \ref{uso} that every principal ideal of $X$ is strongly first countable, and therefore has the CSP (and hence is almost $\sigma$-order complete) due to Corollary \ref{main0}.  Therefore $X$ has the CSP, thus $\mathrm{o}=\sigma\mathrm{o}$, which yields $\mathrm{uo}=\mathrm{u\sigma o}$. We conclude that $\mathrm{u\sigma o}_{X}$ is strongly first countable.\medskip

(ii)$\Rightarrow$(i): By Proposition \ref{uso}, every principal ideal of $X$ is $\sigma$-strongly first countable and almost $\sigma$-order complete.  Hence, by Theorem \ref{main1} every principal ideal of $X$ has the CSP. It follows that $X$ has the CSP, therefore by Proposition \ref{ucsp} $\mathrm{uo}=\mathrm{u\sigma o}$ is strongly first countable on $X$. \medskip

(iv)$\Rightarrow$(v): According to Theorem \ref{main2}, each $I_{v_n}$ has a countable $\pi$-base.  It is easy to see that their union is a $\pi$-base for $X$.\medskip

(vi)$\Rightarrow$(vii): Recall that $X^{u}$ is isomorphic to $C^{\infty}\left(K\right)$, where $K$ is the Stone space of $\Bo_{X}$. Since the latter has a countable $\pi$-base, it follows that it is strongly first countable, hence by Corollary \ref{corrr} $C(K)$ is strongly first countable. Since $\1$ is a weak unit in $C^{\infty}\left(K\right)$, uo convergence there is obtained by unbounding order convergence on $C(K)$ by $\1$.  Since the latter convergence is strongly first countable, it follows from Proposition \ref{fsc-unbdd} that the former is too.
\end{proof}

\begin{rem}
We note that if $X$ has a weak unit $e$, then ${\rm uo}$ is first countable on $X$ if and only if $X$ has CSP, and ${\rm uo}$ is strongly first countable on $X$ if and only if $I_e$ is strongly first countable.\qed
\end{rem}

\begin{rem}\label{sep1}
Note that under the conditions of Theorem \ref{uo}, $X$ has a countable subset whose order adherence is $X$.  Indeed, let $A\subseteq X_{+}$ be a countable $\pi$-base; we claim that $Y:=\spa_{\Q} A$ is $\mathrm{o}$-dense. Indeed, the proof of \cite[Theorem 1.34]{Aliprantis:06} shows that $x=\bigvee\left( Y\cap\left[0,x\right]\right)$, for all $x>0$. Hence, for any $x\in X$, both $x^{+}$ and $x^{-}$ can be order approximated by elements of $Y$; since $Y$ is a subgroup, it follows that it is $\mathrm{o}$-dense.
\qed\end{rem}

\begin{rem}\label{sep2}
Having a countable subset whose order adherence is $X$ does not imply the CSP. Let $X:=\ell_{\8}\left(\left[0,1\right)\right)$, which does not have the CSP. Let $Y$ be the set of all functions $f:\left[0,1\right)\to\Q$ such that there is $n\in\N$ so that $f$ is constant on $\left[\frac{k-1}{n},\frac{k}{n}\right)$, for every $k=1,...,n$. Clearly, $Y$ is countable. It is not hard to show that $Y\cap\left[-n\1,n\1\right]$ is dense in $\left[-n\1,n\1\right]$ with respect to pointwise convergence, for every $n\in\N$. Since on order intervals in $X$ pointwise convergence agrees with order convergence, it follows that the order adherence of $Y$ is $X$.
\qed\end{rem}

\begin{rem}\label{sep3}
The conditions in Proposition \ref{ufc} do not imply existence of a countable subset whose $\mathrm{uo}$-closure is $X$. Let $\left(\Omega,\mu\right)$ be the product of continuum many copies of $[0,1]$ with Lebesgue measure. It is easy to see that $X:=L_{1}\left(\mu\right)$ satisfies condition (ii) in Proposition \ref{ufc}.  Assume that $X$ has a countable subset $Y$ whose $\mathrm{uo}$-closure is $X$. Since $\mathrm{uo}$ convergence is stronger than convergence in measure, it follows that $Y$ is dense with respect to the topology of convergence in measure, which is metrizable. Note that this topology agrees with the norm topology on $\left[-\1,\1\right]$, and so it follows from our assumption that $\left[-\1,\1\right]$ is separable with respect to the norm topology. We thus run into a contradiction, since the coordinate projections constitute a continuum large subset of $\left[-\1,\1\right]$ with equal pairwise distances. For more on separability with respect to convergence in measure, see \cite{Kandic:22}.
\qed\end{rem}

\begin{rem}
Combining Propositions \ref{csp} and \ref{ufc} as well as Theorems \ref{main2} and \ref{uo} we have the following sequence of implications:
\begin{center}
${\rm o}$ -- strong first countability $\Rightarrow$ $\mathrm{uo}$ -- strong first countability $\Rightarrow$ $\mathrm{uo}$ -- first countability $\Rightarrow$ $\mathrm{o}$ -- first countability.
\end{center}
None of these implications are reversible:  $C_0(\R)$ is $\mathrm{uo}$ -- strong first countable but not ${\rm o}$ -- strong first countable, $L_{\8}\left[0,1\right]$ is $\mathrm{uo}$ -- first countable but not $\mathrm{uo}$ -- strong first countable, and, $c_{00}\left(\R\right)$ is $\mathrm{o}$ -- first countable but not $\mathrm{uo}$ -- first countable.

It follows from Remarks \ref{sep1}, \ref{sep2} and \ref{sep3} that
\begin{center}
$\mathrm{uo}$ -- strong first countability $\Rightarrow$ order separability in the sense of adherence $\Rightarrow$ $\mathrm{uo}$ separability in the sense of closure.
\end{center}
However,
\begin{center}
order separability in the sense of adherence $\not\Rightarrow$ $\mathrm{o}$ -- first countability,
\end{center}
and
\begin{center}
$\mathrm{uo}$ -- first countability $\not\Rightarrow$ $\mathrm{uo}$ separability in the sense of closure.\hfill\qed
\end{center}
\end{rem}

\section{Sequential arguments}\label{Sec:  Sequential arguments}

From the point of view of analysis, countability conditions on a topology or convergence space are useful because they give access to sequential arguments.  Indeed, as we observed in Section \ref{Sec:  Sequential}, in a weakly first countable space, hence also in a first countable space, the adherence (closure) of a set agrees with its sequential adherence (closure).  In this section we recall a few other results of this nature, mostly taken from  \cite{Beattie:02}, and state some consequences for the concrete convergence structures studied in Section \ref{Sec:  Locally solid}.  The reader will easily formulate similar results not stated here.

Let $X$ and $Y$ be convergence space.  A function $f:X\to Y$ is \term{sequentially continuous} if $x_n\to x$ in $X$ implies that $f(x_n)\to f(x)$ in $Y$.  In general, a sequentially continuous function may fail to be continuous;  this is true even if both spaces are first countable, see \cite[Example 1.6.12]{Beattie:02}.  A more relevant example for us comes from \cite{Taylor:20}.

\begin{ex}
Let $X$ be a $\rm {ru}$-complete vector lattice, and let $Y$ be an $\rm{ru}$-closed sublattice of $X$. The inclusion map from $Y$ into $X$ is $\rm{ru}$-continuous; in other words, the identity map on $Y$ is $\rm{ru}_{Y}$-to-$\rm{ru}_{X}$ continuous. Furthermore, a sequence $(y_n)$ in $Y$ converges ${\rm ru}$ to $0$ in $Y$ if and only if it converges ${\rm ru}$ to $0$ in $X$, but this is not generally true for nets; that is, the identity map on $Y$ is $\rm{ru}_{X}$-to-$\rm{ru}_{Y}$ sequentially continuous, but not continuous.
\qed\end{ex}


\begin{thm}[cf. \cite{Beattie:02}, Theorem 1.6.14 \& Proposition 1.6.16]
Let $X$ and $Y$ be convergence spaces with $X$ first countable, and $f:X\to Y$ a sequentially continuous function.  If $Y$ is countably sequentially determined, in particular if $Y$ is strongly first countable, then $f$ is continuous.
\end{thm}

\begin{cor}\label{Cor:  Ru Continuity I}
Let $X$ be a first countable convergence space and $Y$ an Archimedean vector lattice.  Assume that there exists a countable set $C\subseteq Y$ so that $I_C=Y$.  If $f:X\to Y$ is sequentially continuous with respect to relative uniform convergence on $Y$, then it is continuous.
\end{cor}

\begin{cor}
Let $X$ and $Y$ be Archimedean vector lattices and $T:X\to Y$ a linear operator.  Consider the following statements. \begin{enumerate}
    \item[(i)] $T$ is order bounded.
    \item[(ii)] $T$ is ${\rm ru}$-to-${\rm ru}$ continuous.
    \item[(iii)] $T$ is sequentially ${\rm ru}$-to-${\rm ru}$ continuous.
\end{enumerate}
Always, (i) and (ii) are equivalent, and these imply (iii).  If there exists a countable set $C\subseteq Y$ so that $I_C=Y$, then all three statements are equivalent.
\end{cor}

\begin{proof}
The equivalence of (i) and (ii) is proven in \cite[Theorem 5.1]{Taylor:20}, and (ii) clearly implies (iii).  The remainder of the result follows from Corollary \ref{Cor:  Ru Continuity I}.
\end{proof}

The reader is invited to formulate an analogous result for the interrelationship between order continuity and sequential order continuity using the results of the preceding section.

\begin{ex}
It is not hard to show that the operator $T:L_{1}[0,1]\to c_{0}$ obtained by integrating against the Rademacher functions is sequentially ${\rm ru}$-to-${\rm ru}$ continuous, but not order bounded, and so not ${\rm ru}$-to-${\rm ru}$ continuous. Note that both Banach lattices in question are order continuous, and so $\rm ru=\rm o=\sigma\rm o$. Hence, the considered operator is also sequentially order- and $\sigma$-order continuous, but not order- or $\sigma$-order continuous.
\qed\end{ex}

Let $X$ be a convergence vector space.  A net $(x_\alpha)_{\alpha \in A}$ in $X$ is a \term{Cauchy net} if the net of differences $(x_\alpha - x_\beta)_{(\alpha,\beta)\in A^2}$ converges to $0$; Cauchy sequences are defined in the same way.  We call $X$ \term{complete} if every Cauchy net in $X$ is convergent, and \term{sequentially complete} if every Cauchy sequence is convergent.  Clearly, every complete space is sequentially complete.  The converse is not true even for Fr\'{e}chet-Urysohn topological vector spaces.

\begin{ex}
Let $Z$ be a Tychonoff space. It is easy to see that $C(Z)$ is pointwise complete if and only if $Z$ is discrete. However, if $Z$ is Lindel\"{o}f $P$-space, then pointwise convergence on $C(Z)$ is sequentially complete and Fr\'{e}chet-Urysohn, see the \cite[proof of Corollary 14.2]{Kakol:11}. An example of such a space which is not discrete is the one point Lindel\"{o}fication of an uncountable discrete space, that is an uncountable set with a selected point $z_{0}$ such that a set $A\subseteq Z$ is closed if it is either at most countable, or contains $z_{0}$.
\qed\end{ex}


For first countable spaces, the situation is better.

\begin{thm}[cf. \cite{Beattie:02}, Proposition 3.6.5]
Let $X$ be a first countable convergence vector space.  Then $X$ is complete if and only if it is sequentially complete.
\end{thm}

\begin{cor}
Let $X$ be an Archimedean vector lattice.  Then the following hold. \begin{enumerate}
    \item[(i)] $X$ is ${\rm ru}$-complete if and only if it is sequentially ${\rm ru}$-complete.
    \item[(ii)] $X$ is $\sigma{\rm o}$-complete if and only if it is sequentially $\sigma{\rm o}$-complete.
    \item[(iii)] If $X$ has the CSP, then $X$ is ${\rm o}$-complete if and only if it is sequentially ${\rm o}$-complete, if and only if it is $\sigma{\rm o}$-complete.
    \item[(iv)] If disjoint sets in $X$ are at most countable, then $X$ is ${\rm uo}$-complete if and only if it is sequentially ${\rm uo}$-complete.
\end{enumerate}
\end{cor}

\begin{question}
Is it true that every sequentially complete weakly first countable convergence vector space is complete?
\end{question}

\subsection*{Acknowledgements}

The authors would like to thank Jurie Conradie, Yang Deng, Vladimir Troitsky and Marten Wortel for valuable discussions. Additionally we credit Will Brian for contributing Proposition \ref{cdp}, and the service \href{mathoverflow.com/}{MathOverflow} which made it possible. We thank the anonymous reviewer for stimulating comments which inspired Propositions \ref{Prop:  Locally solid mod properties} and \ref{Prop:  Locally solid mod countability}. Part of the work on the project was done during the first and the third authors' visit to Mathematical Institute at the Leiden University. The authors are grateful to Marcel de Jeu for hospitality. The first author was supported by Pacific Institute for the Mathematical Sciences.

\end{document}